\documentclass[12pt]{article}

\usepackage{CJK}
\usepackage{latexsym}
\usepackage{amssymb, graphicx, amsmath}
\usepackage{booktabs}
\usepackage{cases}
\usepackage{color}
\definecolor{myblue}{rgb}{0,0,0.5}
\definecolor{mygreen}{rgb}{0,0.5,0}
\definecolor{myred}{rgb}{0.5,0,0}

\usepackage{enumerate}

\usepackage{amsmath,amsthm,amssymb}
\usepackage{graphicx}
\usepackage{subfigure}
\usepackage{caption}
\usepackage{epsfig}
\usepackage{epstopdf}
\usepackage{multirow}
\usepackage{url}
\usepackage{hyperref}

\usepackage{listings}
\usepackage{float}

\usepackage[linesnumbered,ruled,vlined,algo2e]{algorithm2e}

\usepackage{threeparttable}
\usepackage{caption}
\newcommand{\RNum}[1]{\uppercase\expandafter{\romannumeral #1\relax}}

\def \[{\begin{equation}}
\def \]{\end{equation}}

\newtheorem{theorem}{Theorem}[section]

\newtheorem{lemma}{Lemma}[section]

\newtheorem{remark}{Remark}

\begin{CJK}{GBK}{song}
\gdef\cdefinition{定义\,}
\gdef\clemma{引理\,}
\gdef\ctheorem{定理\,}
\end{CJK}

\begin{document}
\begin{CJK*}{GBK}{song}

\begin{center}

{\Large \bf Adaptive primal dual hybrid gradient algorithms based on average spectrum for saddle point problems}\\

\bigskip

 {\bf Shengjie Xu}\footnote{\parbox[t]{16cm}{
 Department of Mathematics, Harbin Institute of Technology, Harbin, China.  This author was supported by the NSFC grant 12501444 and the NSFC grant 72501298.   Email: xsjnsu@163.com
  }}
  \quad 
{\bf Bingsheng He}\footnote{\parbox[t]{16.0cm}{
Department of Mathematics, Nanjing University, China.
  This author was supported by the NSFC Grant
11871029. Email: hebma@nju.edu.cn}}

\medskip

\today

\end{center}

\medskip

{\small

\parbox{0.95\hsize}{

\hrule

\medskip

{\bf Abstract.} The primal dual hybrid gradient algorithm (PDHG), which is also known as the Arrow-Hurwicz method, is a fundamental algorithm for saddle point problems especially in imaging. It also inspires a great number of influential algorithms such as the stochastic PDHG and the Chambolle-Pock's primal dual algorithm. In the literature,  convergence theory of the PDHG is established only when some more restrictive conditions are additionally assumed, and  it is proved that the PDHG with any constant step sizes could diverge for generic setting of convex saddle point problems.  The Chambolle-Pock's primal dual algorithm, as an influential variant of the PDHG, is thus widely used due to its provable convergence theory and competitive numerical performance. However,  step sizes of the Chambolle-Pock's primal dual algorithm are inherently bounded by its associated matrix spectrum, and this restriction could limit its computational capacity structurally. To address these limitations both in theory and practice,  we propose a class of adaptive primal dual hybrid gradient algorithms for generic convex saddle point problems in this paper. By exploiting the prediction-correction algorithmic framework, the global convergence theory of the proposed schemes can be determined only by the average spectrum of the underlying matrix, and it thus leads to a potential acceleration. The numerical experiment on the assignment problem illustrates the superior numerical performance of the proposed method.

\medskip

\noindent {\bf Keywords}:  saddle point problem, primal dual hybrid gradient algorithm,   adaptive parameter tuning, prediction-correction scheme, assignment problem

 \medskip

  \hrule

  }}

\bigskip

\section{Introduction}

This study focuses on the following canonical convex saddle point problem
\begin{equation}\label{Min-Max}
  \min_{x\in \cal{X}} \max_{y\in \cal{Y}} \Phi(x,y) := \theta_1(x) -  y^TAx - \theta_2(y),
\end{equation}
where $\theta_1:\Re^n\to \Re$ and $\theta_2:\Re^m\to \Re$ are proper convex but not necessarily smooth functions, ${\cal X} \subseteq \Re^n$ and ${\cal Y} \subseteq \Re^m$ are closed convex sets,  and $A\in\Re^{m\times n}$ is a given matrix. Throughout our discussion, we denote by $\rho(\cdot)$ the spectrum of a matrix and assume that the solution set of \eqref{Min-Max} is  nonempty. In practice, the saddle point problem \eqref{Min-Max} has captured a large multitude of application problems in various fields. For instance, we can refer the readers to e.g.,  \cite{AHU,Deteix2022,Miroslav2018,QV1997} for some scientific computing models and \cite{CC2010,CHPock,CP-Acta,ZhuChan} for a great number of image restoration problems. In particular, the optimality condition of the canonical convex minimization problem with linear constraints can be reformulated as a special case of the studied model \eqref{Min-Max}.

 To solve the model \eqref{Min-Max}, the Arrow-Hurwicz method originally proposed in \cite{AHU} is fundamental and influential, and it is also known as the primal dual hybrid gradient algorithm (PDHG) emphasized in \cite{ZhuChan} for efficiently solving some variational image reconstruction problems.  As discussed in e.g., \cite{CHPock,HeYuanSIAMIS,ZhuChan}, with the given $(x^k,y^k)$, the iterative scheme of the PDHG can be concretely specified as 
 \begin{subequations} \label{PDA}
\begin{numcases}{\hbox{(PDHG)}}
\label{PDA-x}  x^{k+1} \in \arg\min \Big\{ \Phi(x,y^k) + \frac{r}{2}\|x-x^k\|^2 \;|\; x\in {\cal X} \Big\},\\[0.2cm]
\label{PDA-y}  y^{k+1} \in  \arg\max \Big\{ \Phi(x^{k+1}, y) - \frac{s}{2}\|y-y^k\|^2  \;|\; y\in {\cal Y} \Big\},
\end{numcases}
\end{subequations}
where $r>0$ and $s>0$ are the regularization parameters corresponding to the primal and dual step sizes respectively.  It is clear that the coupled term $y^TAx$ is decoupled by optimizing the primal and dual variables alternatingly, and it thus reduces the computational complexity of the original problem significantly. In the following sections, we assume the subproblems \eqref{PDA-x} and  \eqref{PDA-y} have closed-form solutions or can be solved easily with a high precision.  Due to the ground breaking work \cite{ZhuChan}, the benchmark PDHG has immediately inspired a great number of primal dual type algorithms such as the stochastic PDHG \cite{Alacaoglu 2022,chambolle2018stochastic,Chambolle2024}, the Chambolle-Pock's primal dual algorithm \cite{CHPock,CP-MP,MP2016}  and the generalized primal dual algorithms \cite{EZC,HeMaYuan2017,HeMXY,HeYuanSIAMIS,Zhang}. We also refer the readers to, e.g., \cite{CC2014,Condat,GHY2014,KP2015,Ma2019,PockC} for more variously influential and important variants.

To guarantee the convergence of the PDHG \eqref{PDA}, there are some more restrictive conditions on the functions, domains and step sizes are additionally assumed. For example, when $\theta_1$  is locally Lipschitz continuous and the domain $\mathcal{Y}$ is bounded, it was proved in \cite{BR2012} that the PDHG is convergent provided that the step sizes are vanishing and the sequences of step sizes satisfy some summable conditions. When one of the subfunctions in \eqref{Min-Max} is strongly convex,  the global convergence is established in \cite{HeYouYuan} provided that the step sizes are bounded by $\rho(A^TA)$ as well as the strong convexity modulus of the strongly convex function.  However, for the generic setting of the studied model \eqref{Min-Max}, it was shown in \cite{HXY-AH}  that the PDHG \eqref{PDA} could diverge for any fixed constant step sizes. Consequently, the Chambolle-Pock's primal dual algorithm proposed in \cite{CHPock}, as one of the most influential variants of \eqref{PDA}, is widely used due to its provable convergence theory \cite{CHPock,HeYuanSIAMIS} as well as its efficient numerical performance \cite{CC2010,CP-Acta}. More concretely,  the Chambolle-Pock's primal dual algorithm can be stated as 
\begin{subequations} \label{C-P}
\begin{numcases}{}
\label{C-P-x}  x^{k+1} =\arg\min \big\{\Phi(x,y^k) + \frac{r}{2}\|x-x^k\|^2 \,\;|\;\, x\in {\cal X} \big\},\\[0.1cm]
\label{C-P-x-bar} \bar{x}^{k+1}= 2x^{k+1} -x^k,\\[0.1cm]
\label{C-P-y}
     y^{k+1} = \arg\max \big\{\Phi(\bar{x}^{k+1}, y) - \frac{s}{2}\|y-y^k\|^2  \,\;|\;\, y\in {\cal Y} \big\},
\end{numcases}
\end{subequations}
in which the regularization parameters  $r>0$ and $s>0$ are required to satisfy the condition 
\begin{equation}\label{CP-CC}
 rs>\rho(A^TA)
\end{equation}
to ensure the convergence theory of \eqref{C-P}.  It is clear that relaxing such a conventional condition could result in a potential acceleration, and the recent work \cite{HeMXY,Jiang2023,Yan2024} shows that the condition \eqref{CP-CC} can be further optimally improved to $rs>0.75\rho(A^TA)$.  Since the convergence of the PDHG \eqref{PDA} with fixed step sizes is established only when some more conservative conditions are additionally assumed and the PDHG with any constant step sizes could diverge for the generic setting of the model \eqref{Min-Max}, and that the numerical capacity of Chambolle-Pock's primal dual algorithm \eqref{C-P} is structurally dominated by $\rho(A^TA)$, it is significant to study the primal dual type algorithms with varying step sizes such that the conventional condition \eqref{CP-CC} can be relaxed inherently. 

In fact, as discussed in \cite{HeMXY},  the strong condition \eqref{CP-CC}  essentially ensures the regularization parameters $r$ and $s$ satisfy the bound inequality 
\begin{equation}\label{C-Condition}
 r s\|x^k-x^{k+1}\|^2>\|A(x^k-x^{k+1})\|^2
\end{equation}
for all iterations uniformly, and it thus provides an iteration-independent and hence conservative bounds for the choice of $r s$. If the average spectrum of $A^T\!A$  is significantly smaller than the  spectrum $\rho(A^T\!A)$, an essential question is whether the convergence theory of \eqref{PDA} still maintains provided that the condition \eqref{C-Condition} holds dynamically. That is 
\begin{equation}
  r_k s_k\|x^k-x^{k+1}\|^2>\|A(x^k-x^{k+1})\|^2.
\end{equation}
The preliminary purpose of this paper is to give an affirmative answer in some modified sense and propose a class of adaptive primal dual hybrid gradient algorithms based on the average spectrum for efficiently solving the studied model \eqref{Min-Max}.

The rest of the paper is  organized as follows. Some preparatory notations and results are summarized in Section \ref{sec2},  and an adaptive dual primal hybrid gradient algorithm based on average spectrum of  $A^T\!A$ is presented in Section \ref{sec3}. We further present an adaptive primal dual hybrid gradient scheme based on  average spectrum of $AA^T$ in Section \ref{sec4}. The superior numerical performance of the proposed scheme is demonstrated in Section \ref{sec5}. Finally, some conclusive remarks are given in Section \ref{sec6}.

\section{Preliminaries}\label{sec2}
\setcounter{equation}{0}
\setcounter{remark}{0}

In this section, we summarize some preliminary notations and results for further analysis. A fundamental conclusion which characterizes the variational inequality (VI) structure for a composite convex minimization problem is given first by the following lemma.
\begin{lemma} \label{CP-TF}
\begin{subequations} \label{CP-TF0}
Let $f(z)$ and $g(z)$ be convex functions, and $\mathcal{Z} \subseteq \Re^n$ be a closed convex set. If $g$ is differentiable on an open set which contains $\mathcal{Z}$, and the solution set of the composite minimization problem
   $$
   \min\{f(z) + g(z) \; |\; z\in \mathcal{Z}\}
   $$ is nonempty, then we have
   \begin{equation}\label{CP-TF1}
     z^*  \in \arg\min \{  f(z) + g(z)  \; | \;  z\in \mathcal{Z}\}
   \end{equation}
if and only if
\begin{equation}\label{CP-TF2}
 z^*\in \mathcal{Z}, \quad   f(z) - f(z^*) + (z-z^*)^T\nabla g(z^*) \ge 0, \quad \forall\, z\in \mathcal{Z}.
\end{equation}
\end{subequations}
\end{lemma}
\begin{proof}
See the proof of Theorem 3.1.23 in the monograph \cite{Nesterov2018}.
\end{proof}
With the assertion of Lemma \ref{CP-TF}, we now derive the associated VI representation of the optimality condition for the studied model \eqref{Min-Max}. More concretely, a point pair $(x^*,y^*)\in \mathcal{X}\times \mathcal{Y}$ is called a saddle point  of  \eqref{Min-Max} if it satisfies the inequalities
$$\Phi(x,y^*) \geq \Phi(x^*,y^*) \geq \Phi(x^*,y), \quad \forall\, (x,y)\in\mathcal{X}\times \mathcal{Y},$$
which can be further rewritten as 
$$
    \left\{ \begin{array}{l}
     x^*  \in \arg\min \big\{\Phi(x,y^*)  \,\;|\;\,  x\in \mathcal{X} \big\},   \\[0.2cm]
     y^*\in \arg\max \big\{\Phi(x^*,y) \,\;|\;\, y\in \mathcal{Y} \big\}.
        \end{array} \right.
       $$
Note that the subfunctions $\theta_1$ and $\theta_2$ given in \eqref{Min-Max} are not necessarily smooth. It follows from Lemma \ref{CP-TF} that the saddle point $(x^*, y^*)$ also satisfies 
\begin{equation}\label{VI-Chara}
 \left\{ \begin{array}{lll}
     x^*\in \mathcal{X},  &   \theta_1(x) -  \theta_1(x^*) + (x-x^*)^T(- A^Ty^*) \ge 0, & \forall\; x\in \mathcal{X}, \\[0.1cm]
     y^*\in \mathcal{Y},  &   \theta_2(y) - \theta_2(y^*)  + (y-y^*)^T(Ax^*)\ge 0,  &  \forall \;  y\in \mathcal{Y}.
        \end{array} \right.
\end{equation}
Furthermore, by denoting 
\begin{subequations}\label{VI}
\begin{equation}\label{VI-notations}
 w = \left(\!\begin{array}{c}
                     x\\
                   y \end{array}\! \right),
  \quad \theta(w) = \theta_1(x) + \theta_2(y), \quad
    F(w) =\left(\!\begin{array}{c}
     - A^Ty \\
     Ax \end{array} \!\right)  \quad\hbox{and}\quad \Omega = \mathcal{X} \times \mathcal{Y}, 
\end{equation}
the inequalities above can be compactly rewritten as the following VI:
\begin{equation}\label{VI-problem}
 \hbox{VI}(\Omega, F,\theta): \quad  w^*\in \Omega, \quad \theta(w) -\theta(w^*) + (w-w^*)^T F(w^*) \ge 0, \quad \forall \, w\in\Omega.
\end{equation}
\end{subequations}
Since the operator $F$ in \eqref{VI-notations} is affine with a skew symmetric structure, we have
\begin{equation}\label{EQF}
(w_1- w_2)^T(F(w_1) -F(w_2))\equiv 0, \quad \forall\; w_1, w_2 \in \Re^{(n+m)},
\end{equation}
which indicates that $F$ is also monotone. In the following we denote by $\Omega^*$ the solution set of the VI \eqref{VI}, which is also the saddle point set of the studied model \eqref{Min-Max}.

\section{Adaptive dual primal hybrid gradient algorithm}\label{sec3}
\setcounter{equation}{0}
\setcounter{remark}{0}

In this section, we assume that the primal subproblem of the PDHG can be implemented easily, and we adjust the primal regularization parameter $r_k$ dynamically.  
With this regard, we present an adaptive dual primal  hybrid gradient algorithm based on average spectrum of $A^T\!A$, in which the primal regularization parameter $r_k$ is varying while the dual regularization parameter $s$ is fixed. 

As discussed in, e.g.,  \cite{HeYuanSIAMIS,Ma2019}, the prediction-correction algorithmic framework provides a simple yet powerful analysis tool for simplifying the convergence theory of a certain convex minimization algorithm. In this paper, we follow this methodology and present the novel method also in a prediction-correction manner.

\subsection{Some fundamental matrices and their relationship}

To give the specific scheme of the proposed method, we first define some fundamental matrices which could significantly simplify its convergence theory. Let the prediction matrix  $Q_{k}^{DP}$  and the average norm matrix $H^{DP}$ be defined as
\begin{equation}\label{QDP}
  Q_{k}^{DP}=\left(
               \begin{array}{cc}
                 r_kI_n & 0 \\
                 -A & sI_m \\
               \end{array}
             \right) \quad \hbox{and} \quad   H^{DP}=\left(
               \begin{array}{cc}
                 r_aI_n & 0 \\
                 0 & s I_m \\
               \end{array}
             \right),
\end{equation}
respectively, where 
\begin{equation}
 r_a=\frac{1}{s} \kappa \rho_{\hbox{\footnotesize average}}(A^T\!A)
\end{equation}
with $\kappa>0$ a balanced factor.  Furthermore, we define the correction matrix $M_{k}^{DP}$ as
\begin{equation}\label{MDP}
  M_k^{DP}=\left(
               \begin{array}{cc}
                 \dfrac{r_k}{r_a}I_n & 0 \\[0.4cm]
                 -\frac{1}{s}A & I_m \\
               \end{array}
             \right).
\end{equation}
It is trivial to verify that the tailored matrices defined above satisfy the identity
\begin{equation}\label{DHMQ}
  Q_k^{DP}=H^{DP}M_k^{DP}.
\end{equation}

\subsection{Algorithm}
With the matrices defined above, we now turn to present the adaptive dual primal hybrid gradient algorithm in this subsection. 
The proposed scheme takes the prototype dual primal hybrid gradient algorithm (DPHG)  as a predictor, and then updates the predictor by a simple correction. More concretely, the novel adaptive  method adopts the following prediction-correction scheme.

\begin{center}\fbox{
 \begin{minipage}{16.0cm}
\smallskip
\noindent{\bf{Prediction-correction representation of the adaptive DPHG. }}
\medskip

\textbf{(Prediction Step)} With the given $w^k=(x^k;y^k)$, the predictor $\tilde{w}^k=(\tilde{x}^k;\tilde{y}^k)$ is generated by
 \begin{subequations} \label{ADPA-P}
\begin{numcases}{}
\label{ADPA-y}  \tilde{y}^k \in \arg\max \big\{ \Phi(x^k, y) - \frac{s}{2}\|y-y^k\|^2  \;|\; y\in {\cal Y} \big\}, \\[0.1cm]
\label{ADPA-x}  \tilde{x}^k \in \arg\min \big\{ \Phi(x,\tilde{y}^k) + \frac{r_k}{2}\|x-x^k\|^2 \;|\; x\in {\cal X} \big\}.
\end{numcases}
\end{subequations}

\textbf{(Correction Step) }With the predefined matrices given in \eqref{QDP} and \eqref{MDP}, the new iterate $w^{k+1}=(x^{k+1};y^{k+1})$ is updated by 
\begin{subequations} \label{ADPA-c}
\begin{equation}\label{ADPA-C} 
\left(\!
  \begin{array}{c}
    x^{k+1} \\[0.2cm]
    y^{k+1}
  \end{array}\!
\right)=
\left(\!
  \begin{array}{c}
    x^k \\[0.2cm]
    y^k
  \end{array}\!
\right)-\gamma\alpha_k^\ast
\left(
  \begin{array}{cc}
    \dfrac{r_k}{r_a}I_{n}  & 0 \\[0.1cm]
    -\frac{1}{s}A   & I_m \\
  \end{array}
\right)
\left(\!
  \begin{array}{c}
    x^k-\tilde{x}^k \\[0.2cm]
    y^k-\tilde{y}^k
  \end{array}\!
\right). \qquad \quad
\end{equation}
where
\begin{equation}\label{dp-step}
  \gamma\in(0,2) \quad \hbox{and} \quad  \alpha_k^\ast = \frac{(w^k-\tilde{w}^k)^TQ_k^{DP}(w^k-\tilde{w}^k)}{\|M_k^{DP}(w^k-\tilde{w}^k)\|_{H^{DP}}^2}.
\end{equation}
\end{subequations}
\end{minipage}}
\end{center}

As can be seen easily,  the prediction step \eqref{ADPA-P} essentially shares the same features and computational complexity as the PDHG  scheme \eqref{PDA}, despite their only difference in the order of updating the primal and dual variables. With this regard, we name the proposed scheme \eqref{ADPA-P}-\eqref{ADPA-c} the adaptive DPHG. 

In the following we tune the primal regularization parameter $r_k$ dynamically while keep the dual  regularization parameter $s$ fixed. 
To implement the adaptive DPHG concretely, we first choose a suitable $s$ such that 
$$\|A^Ty-A^Ty^k\|^2\approx s\|y-y^k\|^2.$$
Note that
$$ \|A^T(y^k-\tilde{y}^k)\|^2\approx \rho_{\hbox{\footnotesize average}}(AA^T)\|y^k-\tilde{y}^k\|^2$$
in the sense of probability expectation. We can take 
$$s=\tau\rho_{\hbox{\footnotesize average}}(AA^T)=\frac{1}{m}\tau\hbox{Trace}(AA^T),$$ 
where the balanced factor $\tau$ is empirically suggested to take $\tau\in[1/5,5]$. 
For the fixed $s$, we then adjust the primal regularization parameter $r_k$ dynamically to satisfy the following essential inequality  
\begin{equation}\label{ccdp}
\frac{1}{s}\|A(x^k-\tilde{x}^k)\|^2\leq\nu r_k\|x^k-\tilde{x}^k\|^2 \quad \hbox{with} \quad \nu\in(0,1).
\end{equation}
It is clear that the condition \eqref{ccdp} naturally holds when $\nu r_ks>\rho(A^T\!A)$, which also corresponds to the uniformly conservative condition \eqref{CP-CC} and the desired dynamic condition \eqref{C-Condition} when $\nu\rightarrow1$.  
In practice, a bigger initial regularization parameter $r_0$ could reduce the tuning times for the condition \eqref{ccdp} while lead to a smaller step size, and it is thus significant to choose an appropriate initial $r_0$ to balance the tuning times and step size. In light of  
\begin{equation}\label{aver-con}
  \|A(x^k-\tilde{x}^k)\|^2\approx \rho_{\hbox{\footnotesize average}}(A^T\!A)\|x^k-\tilde{x}^k\|^2,
\end{equation}
we can take  
\begin{equation}\label{r0}
  r_0=\frac{3}{2s}\rho_{\hbox{\footnotesize average}}(A^T\!A)=\frac{3}{2ns}\hbox{Trace}(A^T\!A).
\end{equation}
In fact, combining with the approximate  \eqref{aver-con} and the equality \eqref{r0}, we have 
$$\|A(x^k-\tilde{x}^k)\|^2\approx \rho_{\hbox{\footnotesize average}}(A^T\!A)\|x^k-\tilde{x}^k\|^2\overset{\eqref{r0}}{=}\frac{2}{3}r_0 s\|x^k-\tilde{x}^k\|^2,$$
it further implies that 
$$\frac{1}{s}\|A(x^k-\tilde{x}^k)\|^2\approx \frac{2}{3} r_0\|x^k-\tilde{x}^k\|^2.$$
It thus ensures the relaxed condition \eqref{ccdp} formally.

With the prediction-correction scheme \eqref{ADPA-P}-\eqref{ADPA-c}, the essential step of adaptive DPHG is to design a simple operational manner to satisfy the relaxed condition \eqref{ccdp} dynamically. In the following we give a concrete adaptive tuning criteria based on the average spectrum of $A^T\!A$, and the adaptive DPHG can be concretely summarized as Algorithm \ref{alg2}.
\bigskip

\begin{algorithm2e}[H]
		\SetAlgoLined			% 增添end行
		\DontPrintSemicolon		% 不显示行末尾的分号
        \SetKwInOut{Input}{\textbf{Input}}		% Set the Input
        \SetKwInOut{Output}{\textbf{Output}}	% set the Output
		
		\Input{$k=0$,  $s=\tau\rho_{\hbox{\footnotesize average}}(AA^T)$ with $\tau\in[1/5,5]$, $r_0=\frac{3}{2s}\rho_{\hbox{\footnotesize average}}(A^T\!A)$, \\
$r_a=\kappa \rho_{\hbox{\footnotesize average}}(A^T\!A)$ with $\kappa\in[1/10,10]$, $\underline{r}=\sqrt{\frac{\rho_{\hbox{\scriptsize average}}(A^T\!A)}{\rho(A^T\!A)}}r_a<r_a$, \\
$\gamma=1$, $\theta=1.2$, $\mu=0.5$ and $\nu=0.9$ satisfying $\nu>\mu$.    }
		
        Start with $w^k=(x^k;y^k)$ \tcp*{Initialization}

		\While{the stopping criterium is not satisfied}
		{Step 1: Getting $\tilde{y}^k$ by solving dual subproblem \eqref{ADPA-y}.\;
          Step 2: Getting $\tilde{x}^k$ by solving the primal subproblem \eqref{ADPA-x}.\; \qquad\quad\, Calculating 
          $t=(\frac{1}{s}\|Ax^k-A\tilde{x}^k\|^2)/(r_k\|x^k-\tilde{x}^k\|^2)$.\;
            \While{$t>\nu$}
            {$r_k$ $\gets$ $r_k\times t \times \theta$.\;
             Getting $\tilde{x}^k$ by solving the primal subproblem \eqref{ADPA-x}. \; Calculating 
             $t=(\frac{1}{s}\|Ax^k-A\tilde{x}^k\|^2)/(r_k\|x^k-\tilde{x}^k\|^2).$
            }
            Step 3: Updating the new iterate $w^{k+1}=(x^{k+1};y^{k+1})$ by the correction step \eqref{ADPA-c}. \;
            \emph{Calculating the  stopping criterium.} \;
            Step 4: 		Decreasing $r_k$ if necessary.\;
            \If{$t\leq\mu$  and $r_k>\underline{r}$}
		{$r_{k+1}=\max\big\{r_k\times\frac{2}{3},\,\underline{r}\big\}$; \;
            \Else
		{$r_{k+1}=r_k$.}}
       % Step 5: 		k $\gets$ k + 1\;
		}
		\caption{\textbf{Adaptive dual primal hybrid gradient algorithm for \eqref{Min-Max}} }\label{alg2}
	\end{algorithm2e}

\quad
\begin{remark}
The adaptive DPHG scheme  (i.e., Algorithm \ref{alg2}) is more appropriate for the case where $\rho_{\hbox{\footnotesize average}}(A^T\!A)\ll \rho(A^T\!A)$ and the primal subproblem can be implemented easily.
\end{remark}

For the various parameters in Algorithm \ref{alg2}, we give the specific meaning of each parameter as follows.
\begin{itemize}
  \item $r_k$ represents the regularization parameter of the primal subproblem \eqref{ADPA-x}, and it needs to satisfy the relaxed condition \eqref{ccdp} dynamically.  
  \item $s$ represents the regularization parameter of the dual subproblem \eqref{ADPA-y}, and we suggest  to take $s=\tau\rho_{\hbox{\footnotesize average}}(AA^T)$ with $\tau\in[1/5,5]$ empirically.  
  \item $\gamma\in(0,2)$ represents the relaxation factor of the correction step \eqref{ADPA-c}, and we suggest to take $\gamma\in[0.8,1.6]$.
  \item $\nu\in(0,1)$ is the factor characterizing the closeness between the quadratic terms $\frac{1}{s}\|A(x^k-\tilde{x}^k)\|^2$ and $r_k\|x^k-\tilde{x}^k\|^2$, and we suggest to choose $\nu\in(0.8,1)$.
  \item $\theta$ denotes the increase rate of $r_k$, which also needs to satisfy $\theta>\frac{1}{\upsilon}$ to make  $t\theta>1$.
  \item $\mu\in(0,1)$ is used to control the decrease rate of $r_k$  to avoid $r_k$ is too huge, and we suggest to take $\mu\in(0.2,0.6)$ numerically.
  \item $\underline{r}\in(0,r_a)$ represents the lower bound of the $\{r_k\}$, which guarantees the primal regularization parameter $r_k$ is not too small. For instance, we can choose 
  \begin{equation}\label{lowbd}
    \underline{r}:=\sqrt{\frac{\rho_{\mathrm{\footnotesize average}}(A^T\!A)}{\rho(A^T\!A)}}r_a.
  \end{equation}
\end{itemize}

\begin{remark}
Due to the relaxed condition \eqref{ccdp} naturally holds when  $\nu r_ks>\rho(A^T\!A)$, and that 
$$r_k\leftarrow r_k\times t\times \theta=\frac{1}{s}\theta\frac{1}{\|x^k-\tilde{x}^k\|^2}\|Ax^k-A\tilde{x}^k\|^2\leq\frac{1}{s}\theta\rho(A^TA),$$
we have
\begin{equation}\label{uldp}
  r_k\leq\max\Big\{\frac{2}{s}\theta\rho(A^TA),\frac{1}{\nu s}\rho(A^TA)\Big\}=:\overline{r}.
\end{equation}
This indicates the varying regularization parameter $r_k$ also has a consistent upper bound $\overline{r}$.
\end{remark}

\subsection{Convergence analysis}\label{susus}

In this subsection, we establish the global convergence theory of the adaptive DPHG (i.e., Algorithm \ref{alg2}) based on its prediction-correction representation \eqref{ADPA-P}-\eqref{ADPA-c}. We first summarize the VI characterization of the prediction step \eqref{ADPA-P} by the following lemma. 

\begin{lemma}\label{DP-Q}
Let $Q_k^{DP}$ be the prediction matrix defined in \eqref{QDP} and $\tilde{w}^k=(\tilde{x}^k;\tilde{y}^k)$ be the predictor generated by the DPHG step \eqref{ADPA-P} with the given $w^{k}=(x^{k};y^{k})$. Then, the predictor $\tilde{w}^k$ satisfies the VI
\begin{equation}\label{ADPA-Q}
  \tilde{w}^k\in\Omega, \quad \theta(w) -\theta(\tilde{w}^{k}) + (w-\tilde{w}^{k})^T F(\tilde{w}^{k}) \ge (w-\tilde{w}^{k})^T  Q_k^{DP}(w^k  -\tilde{w}^{k}), \quad  \forall \; w\in  \Omega.
\end{equation}
\end{lemma}
\begin{proof}
For the dual subproblem \eqref{ADPA-y}, according to Lemma \ref{CP-TF}, we have   
$$\tilde{y}^k\in\mathcal{Y},\quad \theta_2(y)- \theta_2(\tilde{y}^k)+(y-\tilde{y}^k)^T\big\{Ax^k+s(\tilde{y}^k-y^k)\big\}\geq0, \quad \forall \; y\in \mathcal{Y},$$
which can be further rewritten as 
\begin{equation}\label{ADPA-Qy}
  \theta_2(y)- \theta_2(\tilde{y}^k)+(y-\tilde{y}^k)^TA\tilde{x}^k\geq (y-\tilde{y}^k)^T\big\{\!-\!A(x^k-\tilde{x}^k)+s(y^k-\tilde{y}^k)\big\}, \quad \forall \; y\in \mathcal{Y}.
\end{equation}
Similarly, it follows from Lemma \ref{CP-TF} that $\tilde{x}^k\in\mathcal{X}$ satisfies the inequality
\begin{equation}\label{ADPA-Qx}
   \theta_1(x)- \theta_1(\tilde{x}^k)+(x-\tilde{x}^k)^T(-A^T\tilde{y}^k)\geq (x-\tilde{x}^k)^Tr_k(x^k-\tilde{x}^k), \quad \forall \; x\in \mathcal{X}.
\end{equation}
By adding \eqref{ADPA-Qy} and \eqref{ADPA-Qx} together, we obtain 
\begin{eqnarray*}
% \nonumber to remove numbering (before each equation)
&&  [\theta_1(x)+\theta_2(y)]-[\theta_1(\tilde{x}^k)+\theta_2(\tilde{y}^k)]+\left(\!
                                                                            \begin{array}{c}
                                                                              x-\tilde{x}^k \\
                                                                              y-\tilde{y}^k \\
                                                                            \end{array}\!
                                                                          \right)^T\left(\!
                                                                                     \begin{array}{c}
                                                                                       -A^T\tilde{y}^k \\
                                                                                       A\tilde{x}^k \\
                                                                                     \end{array}\!
                                                                                   \right) \\[0.1cm]
&&  \;\; \geq\left(\!
                                                                            \begin{array}{c}
                                                                              x-\tilde{x}^k \\
                                                                              y-\tilde{y}^k \\
                                                                            \end{array}\!
                                                                          \right)^T\left(\!\!
                                                                                     \begin{array}{c}
                                                                                       r_k(x^k-\tilde{x}^k) \\
                                                                                       -A(x^k-\tilde{x}^k)+s(y^k-\tilde{y}^k) \\
                                                                                     \end{array}\!\!
                                                                                   \right).
\end{eqnarray*}
The assertion follows immediately by using the notations defined in \eqref{VI} and \eqref{QDP}.
\end{proof}
Throughout our discussion, we assume that $\|w^k-\tilde{w}^k\|\neq0$. Otherwise, $w^k=\tilde{w}^k$ would be a solution point of \eqref{VI} according to \eqref{ADPA-Q}. 
In fact, the right-hand side of \eqref{ADPA-Q} can be further bounded by a simple quadratic terms, which is summarized  by the following lemma.
\begin{lemma}
Let $\{w^k\}$ and $\{\tilde{w}^k\}$ be the sequences generated by Algorithm \ref{alg2} and the relaxed condition \eqref{ccdp} hold. Then, for any $\nu\in(0,1)$, we have
\begin{equation}\label{Key-inequality}
  (w^k-\tilde{w}^k)^TQ_k^{DP}(w^k-\tilde{w}^k)\geq\frac{1}{2}\big\{r_k\|x^k-\tilde{x}^k\|^2+s\|y^k-\tilde{y}^k\|^2\big\}.
\end{equation}
\end{lemma}
\begin{proof}
Recall  $Q_k^{DP}$ defined in \eqref{QDP}. It follows from  Cauchy-Schwarz inequality that
\begin{eqnarray}
% \nonumber to remove numbering (before each equation)
  \lefteqn{(w^k-\tilde{w}^k)^TQ_k^{DP}(w^k-\tilde{w}^k)  =  r_k\|x^k-\tilde{x}^k\|^2 - (y^k-\tilde{y}^k)^TA(x^k-\tilde{x}^k) + s\|y^k-\tilde{y}^k\|^2}     \nonumber \\
 &\geq&   r_k\|x^k-\tilde{x}^k\|^2 -\frac{1}{2}\Big\{\frac{1}{s}\|A(x^k-\tilde{x}^k)\|^2+s\|y^k-\tilde{y}^k\|^2\Big\}+ s\|y^k-\tilde{y}^k\|^2   \nonumber \\
 &\overset{\eqref{ccdp}}{\geq}&   r_k\|x^k-\tilde{x}^k\|^2 -\frac{1}{2}\Big\{\nu r_k\|x^k-\tilde{x}^k\|^2+s\|y^k-\tilde{y}^k\|^2\Big\}+ s\|y^k-\tilde{y}^k\|^2  \\
 &  = &    \frac{1}{2}\big\{(2-\nu)r_k\|x^k-\tilde{x}^k\|^2+s\|y^k-\tilde{y}^k\|^2\big\} \nonumber\\
 &\geq&  \frac{1}{2}\big\{r_k\|x^k-\tilde{x}^k\|^2+s\|y^k-\tilde{y}^k\|^2\big\}. \nonumber
\end{eqnarray}
This completes the proof of the lemma.
\end{proof}

Now we turn to show that $M_k^{DP}(w^k-\tilde{w}^k)$ is indeed an ascent direction of the unknown distance function $\|w-w^\ast\|_{H^{DP}}^2$ at the point $w^k$. 
To this end, by setting $w$ in \eqref{ADPA-Q} as any $w^\ast\in\Omega^\ast$ and using the identity \eqref{EQF}, we have 
\begin{eqnarray*}
% \nonumber to remove numbering (before each equation)
  \lefteqn{(\tilde{w}^{k}-w^\ast)^T Q_k^{DP}(w^k  -\tilde{w}^{k})\geq \theta(\tilde{w}^{k})-\theta(w^\ast)  + (\tilde{w}^{k}-w^\ast)^T F(\tilde{w}^{k})  } \\
   &=&     \theta(\tilde{w}^{k})-\theta(w^\ast)  + (\tilde{w}^{k}-w^\ast)^T F(w^\ast)\geq0. \qquad \qquad
\end{eqnarray*}
Since $\tilde{w}^{k}-w^\ast=(w^k-w^\ast)-(w^k-\tilde{w}^{k})$,  it further implies that
\begin{eqnarray}\label{K-Dirdp}
% \nonumber to remove numbering (before each equation)
  \lefteqn{ (w^k-w^\ast)^T Q_k^{DP}(w^k  -\tilde{w}^{k})}  \nonumber \\
  &\geq& (w^k-\tilde{w}^k)^TQ_k^{DP}(w^k  -\tilde{w}^{k})\overset{\eqref{Key-inequality}}{\geq}\frac{1}{2}\big\{r_k\|x^k-\tilde{x}^k\|^2+s\|y^k-\tilde{y}^k\|^2\big\}.
\end{eqnarray}
Note that the matrices $Q_k^{DP}$ and $H^{DP}$ defined in \eqref{QDP} are nonsingular. The inequality \eqref{K-Dirdp} can be further reformulated as 
\begin{eqnarray}\label{AD}
% \nonumber to remove numbering (before each equation)
\lefteqn{\Big\langle \nabla(\frac{1}{2}\|w-w^\ast\|_{H^{DP}}^2)\Big|_{w=w^k}, (H^{DP})^{-1}Q_k^{DP}(w^k  -\tilde{w}^{k})\Big\rangle} \nonumber\\
  &\overset{\eqref{DHMQ}}{=}& \Big\langle H^{DP}(w^k-w^\ast), M_k^{DP}(w^k  -\tilde{w}^{k})\Big\rangle  \geq\frac{1}{2}\big\{r_k\|x^k-\tilde{x}^k\|^2+s\|y^k-\tilde{y}^k\|^2\big\},
\end{eqnarray}
which indicates that $M_k^{DP}(w^k  -\tilde{w}^{k})$ is an ascent direction of the unknown distance function $\|w-w^\ast\|_{H^{DP}}^2$ at the point $w^k$. Consequently, we can generate the new iterate $w^{k+1}$ by the correction step
\begin{equation}\label{csdp}
 w^{k+1} = w^k - \alpha M_k^{DP}(w^k  -\tilde{w}^{k})
\end{equation}
with $\alpha>0$ the corresponding correction step size. As discussed in, e.g., \cite{HeYuanSIAMIS}, such a correction step could yield the contraction of proximity to the solution set of  \eqref{Min-Max} if an appropriate step size is taken.   

We now turn to determine the step size $\alpha$ in the correction step \eqref{csdp} to make the new iterate $w^{k+1}$
 closer to $\Omega^\ast$ as much as possible.  Note that
\begin{eqnarray}\label{cc-dp}
% \nonumber to remove numbering (before each equation)
  \lefteqn{\|w^{k+1}-w^\ast\|_{H^{DP}}^2} \nonumber \\
  &=& \|w^k - \alpha M_k^{DP}(w^k  -\tilde{w}^{k})-w^\ast\|_{H^{DP}}^2\\
  &=& \|w^k -w^\ast\|_{H^{DP}}^2-2\alpha(w^{k}-w^\ast)^TH^{DP}M_k^{DP}(w^k  -\tilde{w}^{k})+\alpha^2\|M_k^{DP}(w^k  -\tilde{w}^{k})\|_{H^{DP}}^2 \nonumber \\
  &\overset{\eqref{K-Dirdp}}{\leq}& \|w^k -w^\ast\|_{H^{DP}}^2-2\alpha(w^{k}-\tilde{w}^k)^TQ_k^{DP}(w^k  -\tilde{w}^{k})+\alpha^2\|M_k^{DP}(w^k  -\tilde{w}^{k})\|_{H^{DP}}^2. \nonumber
\end{eqnarray}
Let
$$q_k^{DP}(\alpha):=2\alpha(w^{k}-\tilde{w}^k)^TQ_k^{DP}(w^k  -\tilde{w}^{k})-\alpha^2\|M_k^{DP}(w^k  -\tilde{w}^{k})\|_{H^{DP}}^2.$$
By maximizing the quadratic term $q_k^{DP}(\alpha)$, we have
\begin{equation}\label{dp-stepsize}
  \alpha_k^\ast =\frac{(w^{k}-\tilde{w}^k)^TQ_k^{DP}(w^k  -\tilde{w}^{k})}{\|M_k^{DP}(w^k  -\tilde{w}^{k})\|_{H^{DP}}^2}.
\end{equation}
Owing to $q_k^{DP}(\alpha)$ is a lower bound for some contraction function, we can follow the similar technique in \cite{HeYuanSIAMIS} and introduce a relaxation factor $\gamma\in (0, 2)$ and set $\alpha_k=\gamma\alpha_k^\ast$, which immediately implies the tailored correction step \eqref{ADPA-c}.

The following lemma further characterize the contraction amount of the new iterate.

\begin{lemma}\label{lemma4}
Let $\{w^k\}$ and $\{\tilde{w}^k\}$ be the sequences generated by Algorithm \ref{alg2} and the relaxed condition \eqref{ccdp} hold, and let $\underline{r}$ be a lower bound of the sequence $\{r_k\}$. Then, for any $\gamma\in(0,2)$ and $w^\ast\in\Omega^\ast$, there exists a constant $C^{DP}>0$ such as
\begin{equation}\label{con-con-dp}
  \|w^{k+1}-w^\ast\|_{H^{DP}}^2\leq  \|w^k -w^\ast\|_{H^{DP}}^2-\gamma(2-\gamma)C^{DP}\|w^k-\tilde{w}^k\|_{\underline{H}^{DP}}^2,
\end{equation}
where
\begin{equation}\label{Hldp}
\underline{H}^{DP}=\left(
  \begin{array}{cc}
    \underline{r} I_n  & 0 \\
   0   & sI_m \\
  \end{array}
\right).
\end{equation}
\end{lemma}
\begin{proof}
To begin with, by setting $\alpha=\gamma\alpha_k^\ast$ in \eqref{cc-dp}, we have 
\begin{eqnarray*}
% \nonumber to remove numbering (before each equation)
  \lefteqn{\|w^{k+1}-w^\ast\|_{H^{DP}}^2}\\
  &\leq& \|w^k -w^\ast\|_{H^{DP}}^2-2\gamma\alpha_k^\ast(w^{k}-\tilde{w}^k)^TQ_k^{DP}(w^k  -\tilde{w}^{k})+\gamma^2(\alpha_k^\ast)^2\|M_k^{DP}(w^k  -\tilde{w}^{k})\|_{H^{DP}}^2 \\
  &\overset{\eqref{dp-stepsize}}{=}&\|w^k -w^\ast\|_{H^{DP}}^2-2\gamma\alpha_k^\ast(w^{k}-\tilde{w}^k)^TQ_k^{DP}(w^k  -\tilde{w}^{k})+\gamma^2\alpha_k^\ast(w^{k}-\tilde{w}^k)^TQ_k^{DP}(w^k  -\tilde{w}^{k}) \\
  &=& \|w^k -w^\ast\|_{H^{DP}}^2 - \gamma(2-\gamma)\alpha_k^\ast(w^{k}-\tilde{w}^k)^TQ_k^{DP}(w^k  -\tilde{w}^{k}).
\end{eqnarray*}
Furthermore, it follows from  \eqref{Key-inequality} that 
\begin{eqnarray*}
% \nonumber to remove numbering (before each equation)
  &&\|w^k -w^\ast\|_{H^{DP}}^2-\|w^{k+1}-w^\ast\|_{H^{DP}}^2  \geq  \gamma(2-\gamma)\alpha_k^\ast(w^{k}-\tilde{w}^k)^TQ_k^{DP}(w^k  -\tilde{w}^{k})  \\
   &&\quad \overset{\eqref{Key-inequality}}{\geq} \frac{1}{2}\gamma(2-\gamma)\alpha_k^\ast\big\{r_k\|x^k-\tilde{x}^k\|^2+s\|y^k-\tilde{y}^k\|^2\big\} \\
   &&\quad\;\; \geq\frac{1}{2} \gamma(2-\gamma)\alpha_k^\ast \big\{\underline{r}\|x^k-\tilde{x}^k\|^2+s\|y^k-\tilde{y}^k\|^2\big\}  \\
   &&\quad\;\; =\frac{1}{2}\gamma(2-\gamma)\alpha_k^\ast \|w^k-\tilde{w}^k\|_{\underline{H}^{DP}}^2.
\end{eqnarray*}
Then, it suffices to show that there exists a constant $C^{DP}$ satisfying $\alpha_k^\ast\geq 2C^{DP}$ uniformly. Recall the matrices $H^{DP}$ and $M_k^{DP}$ defined in \eqref{QDP} and \eqref{MDP} respectively. Since $\{r_k\}$ has an upper bound $\overline{r}$ given by \eqref{uldp},  we have
$$0\prec (M_k^{DP})^TH^{DP}M_k^{DP}=\left(\!\!
  \begin{array}{cc}
    \dfrac{r_k^2}{r_a}I_{n} + \frac{1}{s} A^T\!A & -A^T \\
   - A   & sI_m \\
  \end{array}\!\!
\right)\preceq\left(\!\!
  \begin{array}{cc}
    \dfrac{\overline{r}^2}{r_a}I_{n} + \frac{1}{s} A^T\!A & -A^T \\
   - A   & sI_m \\
  \end{array}\!\!
\right)=:\overline{H}^{DP}.$$
It further implies that
\begin{eqnarray*}
% \nonumber to remove numbering (before each equation)
 \alpha_k^\ast &=&\frac{(w^{k}-\tilde{w}^k)^TQ_k^{DP}(w^k  -\tilde{w}^{k})}{\|M_k^{DP}(w^k  -\tilde{w}^{k})\|_{H^{DP}}^2}
  \geq \frac{\|w^k-\tilde{w}^k\|_{\underline{H}^{DP}}^2}{2\|w^k  -\tilde{w}^{k}\|_{\overline{H}^{DP}}^2}.
\end{eqnarray*}
The assertion of lemma follows immediately by using the norm equivalence principle.
\end{proof}

With the assertion of Lemma \ref{lemma4}, we can immediately obtain the global convergence of the proposed adaptive DPHG, which is summarized as the following theorem.
\begin{theorem} 
Let $\{w^k \}$ and $\{\tilde{w}^k\}$  be the sequences generated by Algorithm \ref{alg2} and the relaxed condition \eqref{ccdp} hold. Then, the sequence  $\{w^k \}$  converges to some $w^\infty\in\Omega^\ast$.
\end{theorem}
\begin{proof}
To begin with, it follows from \eqref{con-con-dp} that the generated sequence $\{w^k\}$ is bounded. Let $w^{\infty}$ be a cluster point of $\{w^k\}$ and $\{w^{k_j}\}$ be a subsequence converging to $w^{\infty}$.  By summing the inequality  \eqref{con-con-dp} over $k=0,1,\ldots,\infty$, we obtain
$$\sum_{k=0}^{\infty}\|w^k -\tilde{w}^k\|_{\underline{H}^{DP}}^2\leq\frac{1}{\gamma(2-\gamma)C^{DP}}\|w^0 -w^\ast\|_{H^{DP}}^2,$$
which further implies
\begin{equation}\label{Contrac-uut-dp}
  \lim_{k\to \infty}\|w^k -\tilde{w}^k\|_{\underline{H}^{DP}}=0.
\end{equation}
Moreover, it follows from \eqref{Contrac-uut-dp} that the sequence $\{\tilde{w}^{k_j}\}$ also converges to $w^{\infty}$. Then, according to \eqref{ADPA-Q}, we have
$$\tilde{w}^{k_j}\in \Omega, \quad \theta(w)-\theta(\tilde{w}^{k_j}) + (w-\tilde{w}^{k_j})^TF(\tilde{w}^{k_j}) \ge (w-\tilde{w}^{k_j})^TQ_k^{DP}(w^{k_j}-\tilde{w}^{k_j}), \quad \forall\; w\in \Omega.$$
Using the continuity of $\theta(w)$ and $F(w)$, we further obtain
$$w^{\infty}\in \Omega, \quad \theta(w)-\theta(w^{\infty}) + (w- w^{\infty})^T F(w^{\infty}) \ge 0, \quad \forall\; w\in \Omega, $$
which means  $w^{\infty}$ is a solution point of the VI \eqref{VI}.  Furthermore, it follows from \eqref{con-con-dp} that
$$\|w^{k+1} - w^{\infty}\|_{H^{DP}} \le \|w^k  - w^{\infty}\|_{H^{DP}},$$
which means that the sequence $\left\{\left\|w^{k}-w^{\infty}\right\|_{H^{DP}}\right\}_{k \geq 0}$ is nonincreasing and  it is thus convergent. Since  $\lim_{j\rightarrow\infty}\|w^{k_j}-w^{\infty}\|_{H^{DP}}=0$, we obtain that the sequence $\{w^k\}$ also converges to $w^\infty$. This completes the proof of the theorem.
\end{proof}

\section{Adaptive primal dual hybrid gradient algorithm}\label{sec4}
\setcounter{equation}{0}
\setcounter{remark}{0}

In this section, we assume that the dual subproblem of the PDHG can be implemented easily  and extend the proposed adaptive DPHG to the way that the primal regularization parameter $r$ is fixed while the dual regularization parameter $s_k$ is varying. 
\subsection{Some fundamental matrices and their relationship}
Similarly, let us first define some fundamental matrices for further analysis. The prediction matrix  $Q_{k}^{PD}$  and the average norm matrix $H^{PD}$ are defined as
\begin{equation}\label{QPD}
  Q_{k}^{PD}=\left(
               \begin{array}{cc}
                 rI_n & A^T \\
                 0 & s_kI_m \\
               \end{array}
             \right) \quad \hbox{and} \quad   H^{PD}=\left(
               \begin{array}{cc}
                 rI_n & 0 \\
                 0 & s_aI_m \\
               \end{array}
             \right)
\end{equation}
respectively, where 
\begin{equation}
 s_a:=\kappa\frac{1}{r} \rho_{\hbox{\footnotesize average}}(AA^T)
\end{equation}
with $\kappa>0$ a balanced factor. Furthermore, we define the correction matrix $M_{k}^{PD}$ as
\begin{equation}\label{MPD}
  M_k^{PD}=\left(
               \begin{array}{cc}
                 I_n & \frac{1}{r}A^T \\[0.2cm]
                 0 & \frac{s_k}{s_a}I_m \\
               \end{array}
             \right).
\end{equation}
It is trivial to verify that the matrices defined above satisfy the identity
\begin{equation}\label{HMQ}
  Q_k^{PD}=H^{PD}M_k^{PD}.
\end{equation}

\subsection{Algorithm}
By exploiting the prediction-correction algorithmic framework, the novel adaptive primal dual hybrid gradient algorithm (PDHG)  takes the following fundamental scheme.
\begin{center}\fbox{
 \begin{minipage}{16.0cm}
\smallskip
\noindent{\bf{Prediction-correction representation of the adaptive PDHG. }}
\medskip

\textbf{(Prediction Step)} With the given $w^k=(x^k;y^k)$, we generate the predictor $\tilde{w}^k=(\tilde{x}^k;\tilde{y}^k)$ via
 \begin{subequations} \label{APDA-P}
\begin{numcases}{}
\label{APDA-x}  \tilde{x}^k \in \arg\min \big\{ \Phi(x,y^k) + \frac{r}{2}\|x-x^k\|^2 \;|\; x\in {\cal X} \big\},\\[0.1cm]
\label{APDA-y}  \tilde{y}^k \in \arg\max \big\{ \Phi(\tilde{x}^k, y) - \frac{s_k}{2}\|y-y^k\|^2  \;|\; y\in {\cal Y} \big\}.
\end{numcases}
\end{subequations}
\textbf{(Correction Step)} With the given matrices defined in \eqref{QPD} and \eqref{MPD}, the new iterate $w^{k+1}=(x^{k+1};y^{k+1})$ is updated by 
 \begin{subequations} \label{APDA-c}
\begin{equation}\label{APDA-C}
\left(\!
  \begin{array}{c}
    x^{k+1} \\[0.2cm]
    y^{k+1}
  \end{array}\!
\right)=
\left(\!
  \begin{array}{c}
    x^k \\[0.2cm]
    y^k
  \end{array}\!
\right)-\gamma\alpha_k^\ast\left(
               \begin{array}{cc}
                 I_n & \frac{1}{r}A^T \\[0.2cm]
                 0 & \frac{s_k}{s_a}I_m \\
               \end{array}
             \right)
\left(\!
  \begin{array}{c}
    x^k-\tilde{x}^k \\[0.2cm]
    y^k-\tilde{y}^k
  \end{array}\!
\right),
\end{equation}
where
\begin{equation}\label{pd-step}
  \gamma\in(0,2) \quad \hbox{and} \quad  \alpha_k^\ast = \frac{(w^k-\tilde{w}^k)^TQ_k^{PD}(w^k-\tilde{w}^k)}{\|M_k^{PD}(w^k-\tilde{w}^k)\|_{H^{PD}}^2}.
\end{equation}
\end{subequations}
\end{minipage}}
\end{center}

In contrast to the  prediction-correction algorithmic scheme \eqref{ADPA-P}-\eqref{ADPA-c}, the extended method first adopts the prototype PDHG \eqref{PDA} as a predictor, followed by a simple correction step \eqref{APDA-c}, and they both enjoys the same features and computational complexity. Owing to the extended method takes the PDHG  \eqref{PDA} as a predictor, we call briefly the novel algorithm the adaptive PDHG. To implement the adaptive PDHG concretely, we first choose a fixed primal regularization parameter $r$ to satisfy the approximation 
$$\|Ax-Ax^k\|^2\approx r\|x-x^k\|^2.$$
Since  
\begin{equation}\label{aver-con1}
  \|A(x^k-\tilde{x}^k)\|^2\approx \rho_{\hbox{\footnotesize average}}(A^T\!A)\|x^k-\tilde{x}^k\|^2
\end{equation}
in the sense of probability expectation, we can choose  
$$r=\tau\rho_{\hbox{\footnotesize average}}(A^T\!A)=\frac{1}{n}\tau\hbox{Trace}(A^T\!A),$$ 
in which the balanced factor $\tau$ is empirically suggested to take $\tau\in[1/5,5]$. 

For the chosen $r$,  we then adjust the dual regularization parameter $s_k$ dynamically to satisfy the dynamic condition
\begin{equation}\label{ccpd}
\frac{1}{r}\|A^T(y^k-\tilde{y}^k)\|^2\leq\nu s_k\|y^k-\tilde{y}^k\|^2 \;\;\hbox{with}\;\; \nu\in(0,1).
\end{equation}
Similarly, due to
$$ \|A^T(y^k-\tilde{y}^k)\|^2\approx \rho_{\hbox{\footnotesize average}}(AA^T)\|y^k-\tilde{y}^k\|^2.$$
We can set the initial parameter $s_0$ as
\begin{equation}\label{s0}
  s_0=\frac{3}{2r}\rho_{\hbox{\footnotesize average}}(AA^T)=\frac{3}{2rm}\hbox{Trace}(AA^T).
\end{equation} 
Note that
$$\|A^T(y^k-\tilde{y}^k)\|^2\approx \rho_{\hbox{\footnotesize average}}(AA^T)\|y^k-\tilde{y}^k\|^2\overset{\eqref{s0}}{=}\frac{2}{3} rs_0\|y^k-\tilde{y}^k\|^2.$$
It implies that
$$\frac{1}{r}\|A^T(y^k-\tilde{y}^k)\|^2\approx \frac{2}{3} s_0\|y^k-\tilde{y}^k\|^2,$$
which indeed corresponds to the relaxed condition \eqref{ccpd} formally. Based on the dynamic condition \eqref{ccpd}, the adaptive PDHG is specified as the following scheme. 

%an it thus provides an initial condition to characterize the condition \eqref{ccpd}. Based on the dynamic condition \eqref{ccpd}, the adaptive PDHG can be summarized as the following scheme. 
%the varying condition \eqref{ccdp} is more relaxed than the conservative condition \eqref{CP-CC}, which could result in a smaller regularization  $r_k$ and thus a potential acceleration.
\vspace{0.1cm}
\begin{algorithm2e}[H]
		\SetAlgoLined			% 增添end行
		\DontPrintSemicolon		% 不显示行末尾的分号
        \SetKwInOut{Input}{\textbf{Input}}		% Set the Input
        \SetKwInOut{Output}{\textbf{Output}}	% set the Output
		
		\Input{$k=0$,  $r=\tau\rho_{\hbox{\footnotesize average}}(A^T\!A)$ with $\tau\in[1/5,5]$, $s_0=\frac{3}{2r}\rho_{\hbox{\footnotesize average}}(A\!A^T)$, \\
        $s_a=\kappa \rho_{\hbox{\footnotesize average}}(AA^T)$ with $\kappa\in[1/10,10]$, $\underline{s}=\sqrt{\frac{\rho_{\hbox{\scriptsize average}}(AA^T)}{\rho(AA^T)}}s_a<s_a$, \\ $\gamma=1$, $\theta=1.2$, $\mu=0.5$ and $\nu=0.9$ satisfying $\nu>\mu$.   }
		
        Start with $w^k=(x^k;y^k)$ \tcp*{Initialization}

		\While{the stopping criterium is not satisfied}
		{Step 1: Getting $\tilde{x}^k$ by solving the primal subproblem \eqref{APDA-x}.\;
          Step 2: Getting $\tilde{y}^k$ by solving the dual subproblem \eqref{APDA-y}.\; \qquad\quad\, Calculating 
          $t=(\frac{1}{r}\|A^Ty^k-A^T\tilde{y}^k\|^2)/(s_k\|y^k-\tilde{y}^k\|^2)$.\;
            \While{$t>\nu$}
            {$s_k$ $\gets$ $s_k\times t \times \theta$.\;
             Getting $\tilde{y}^k$ by solving the dual subproblem \eqref{APDA-y}. \; Calculating 
             $t=(\frac{1}{r}\|A^Ty^k-A^T\tilde{y}^k\|^2)/(s_k\|y^k-\tilde{y}^k\|^2).$
            }
            Step 3: Updating the new iterate $w^{k+1}=(x^{k+1};y^{k+1})$ by the correction step \eqref{APDA-c}. \;
            \emph{Calculating the  stopping criterium.} \;
            Step 4: 	Decreasing $s_k$ if necessary.\;
            \If{$t\leq\mu$ and $s_k>\underline{s}$}
		{$s_{k+1}=\max\big\{s_k\times\frac{2}{3},\,\underline{s}\big\}$; \;
            \Else
		{$s_{k+1}=s_k$.}}
       % Step 5: 		k $\gets$ k + 1\;
		}
		\caption{\textbf{Adaptive primal dual hybrid gradient algorithm  for \eqref{Min-Max}} }\label{alg1}
	\end{algorithm2e}

\quad
\begin{remark}
The proposed adaptive PDHG (i.e., Algorithm \ref{alg1}) is more suitable for the case where $\rho_{\hbox{\footnotesize average}}(AA^T)\ll \rho(AA^T)$ and the dual subproblem can be implemented easily.
\end{remark}

For the various parameters in Algorithm \ref{alg1}, we also give the concrete meaning of each parameter as follows.
\begin{itemize}
  \item $r$ represents the  regularization parameter of the primal subproblem \eqref{APDA-x}, and we suggest to take $r=\tau\rho_{\hbox{\footnotesize average}}(A^T\!A)$ with $\tau\in[1/5,5]$. 
  \item $s_k$ represents the regularization parameter of the dual subproblem \eqref{APDA-y}, and it needs to satisfy the relaxed condition \eqref{ccpd} dynamically. 
   \item $\gamma\in(0,2)$ represents the relaxation factor of the correction step \eqref{APDA-c}, and we suggest to take $\gamma\in[0.8,1.6]$ in parctice.
  \item $\nu\in(0,1)$ is the parameter characterizing the closeness between the quadratic terms $\frac{1}{r}\|A^T(y^k-\tilde{y}^k)\|^2$ and $s_k\|y^k-\tilde{y}^k\|^2$, and we suggest to choose $\nu\in(0.8,1)$.
  \item $\mu\in(0,1)$ is used to dynamically control the decrease of $s_k$ to avoid $s_k$ is too huge, and we suggest to take $\mu\in(0.2,0.6)$.
  \item $\theta$ determines the increase rate of $s_k$, which needs to satisfy $\theta>\frac{1}{\upsilon}$ to make  $t\theta>1$.
  \item $\underline{s}\in(0,s_a)$ represents the lower bound of  $\{s_k\}$, which guarantees the regularization parameter $s_k$ is not too tiny. In parctice, we can choose 
  \begin{equation}\label{lowbd-PD}
    \underline{s}:=\sqrt{\frac{\rho_{\mathrm{\footnotesize average}}(AA^T)}{\rho(AA^T)}}s_a.
  \end{equation}
\end{itemize}

\begin{remark}
Since the relaxed condition \eqref{ccpd} naturally holds when  $\nu rs_k>\rho(AA^T)$ and 
$$s_k\leftarrow s_k\times t\times \theta=\frac{1}{r}\theta\frac{1}{\|y^k-\tilde{y}^k\|^2}\|A^Ty^k-A^T\tilde{y}^k\|^2\leq\frac{1}{r}\theta\rho(AA^T),$$
it implies that
\begin{equation}\label{ulpd}
  s_k\leq\max\Big\{\frac{2}{r}\theta\rho(AA^T),\frac{1}{\nu r}\rho(AA^T)\Big\}=:\overline{s}.
\end{equation}
The varying dual regularization parameter $s_k$ hence has an uniformly upper bound $\overline{s}$.
\end{remark}

\subsection{Convergence analysis}
Following the similar analysis route in Subsection \ref{susus}, we establish the global convergence theory of the proposed adaptive PDHG (i.e., Algorithm \ref{alg1}) based on its prediction-correction representation \eqref{APDA-P}-\eqref{APDA-c}. We first specify the VI characterization of the prediction step \eqref{APDA-P} by the following lemma. 
\begin{lemma}\label{PD-Q}
Let $Q_k^{PD}$ be the prediction matrix defined in \eqref{QPD} and $\tilde{w}^k=(\tilde{x}^k;\tilde{y}^k)$ be the predictor generated by the PDHG prediction step \eqref{APDA-P} with the given $w^{k}=(x^k;y^k)$. Then, the predictor $\tilde{w}^k$ satisfies the VI
\begin{equation}\label{APDA-Q}
  \tilde{w}^k\in\Omega, \quad \theta(w) -\theta(\tilde{w}^{k}) + (w-\tilde{w}^{k})^T F(\tilde{w}^{k}) \ge (w-\tilde{w}^{k})^T  Q_k^{PD}(w^k  -\tilde{w}^{k}), \quad  \forall \; w\in  \Omega.
\end{equation}
\end{lemma}
\begin{proof}
For the primal subproblem \eqref{APDA-x}, it follows from Lemma \ref{CP-TF} that $\tilde{x}^k\in\mathcal{X}$ satisfies the variational inequality 
$$\theta_1(x)- \theta_1(\tilde{x}^k)+(x-\tilde{x}^k)^T\big\{-A^Ty^k+r(\tilde{x}^k-x^k)\big\}\geq0, \quad \forall \; x\in \mathcal{X},$$
which can be further rewritten as
\begin{equation}\label{APDA-Qx}
  \theta_1(x)- \theta_1(\tilde{x}^k)+(x-\tilde{x}^k)^T(-A^T\tilde{y}^k)\geq (x-\tilde{x}^k)^T\big\{r(x^k-\tilde{x}^k)+A^T(y^k-\tilde{y}^k)\big\}, \quad \forall \; x\in \mathcal{X}.
\end{equation}
Similarly, the dual variable $\tilde{y}^k$ given by \eqref{APDA-y} satisfies the inequality
\begin{equation}\label{APDA-Qy}
 \tilde{y}^k\in\mathcal{Y},\quad  \theta_2(y)- \theta_2(\tilde{y}^k)+(y-\tilde{y}^k)^TA\tilde{x}^k\geq (y-\tilde{y}^k)^Ts_k(y^k-\tilde{y}^k), \quad \forall \; y\in \mathcal{Y}.
\end{equation}
Combining the inequality \eqref{APDA-Qx} with \eqref{APDA-Qy} together, we have 
\begin{eqnarray*}
% \nonumber to remove numbering (before each equation)
&&  [\theta_1(x)+\theta_2(y)]-[\theta_1(\tilde{x}^k)+\theta_2(\tilde{y}^k)]+\left(\!
                                                                            \begin{array}{c}
                                                                              x-\tilde{x}^k \\
                                                                              y-\tilde{y}^k \\
                                                                            \end{array}\!
                                                                          \right)^T\left(\!
                                                                                     \begin{array}{c}
                                                                                       -A^T\tilde{y}^k \\
                                                                                       A\tilde{x}^k \\
                                                                                     \end{array}\!
                                                                                   \right) \\[0.1cm]
&&  \;\; \geq\left(\!
                                                                            \begin{array}{c}
                                                                              x-\tilde{x}^k \\
                                                                              y-\tilde{y}^k \\
                                                                            \end{array}\!
                                                                          \right)^T\left(\!\!
                                                                                     \begin{array}{c}
                                                                                       r(x^k-\tilde{x}^k)+A^T(y^k-\tilde{y}^k) \\
                                                                                       s_k(y^k-\tilde{y}^k) \\
                                                                                     \end{array}\!\!
                                                                                   \right).
\end{eqnarray*}
The assertion of the lemma follows immediately by using these notations defined in \eqref{VI} and \eqref{QPD}.
\end{proof}
\begin{lemma}
Let $\{w^k\}$ and $\{\tilde{w}^k\}$ be the sequences generated by Algorithm \ref{alg1} and the condition \eqref{ccpd} hold. Then, for any $\nu\in(0,1)$, we have
\begin{equation}\label{Key-inequality1}
  (w^k-\tilde{w}^k)^TQ_k^{PD}(w^k-\tilde{w}^k)\geq\frac{1}{2}\big\{r\|x^k-\tilde{x}^k\|^2+s_k\|y^k-\tilde{y}^k\|^2\big\}.
\end{equation}
\end{lemma}
\begin{proof}
Recall the prediction matrix $Q_k^{PD}$ defined in \eqref{QPD}. It follows from  Cauchy-Schwarz inequality that
\begin{eqnarray}
% \nonumber to remove numbering (before each equation)
  \lefteqn{(w^k-\tilde{w}^k)^TQ_k^{PD}(w^k-\tilde{w}^k)  =  r\|x^k-\tilde{x}^k\|^2 + (x^k-\tilde{x}^k)^TA^T(y^k-\tilde{y}^k) + s_k\|y^k-\tilde{y}^k\|^2}     \nonumber \\
 &\geq&   r\|x^k-\tilde{x}^k\|^2 -\frac{1}{2}\Big\{\frac{1}{r}\|A^T(y^k-\tilde{y}^k)\|^2+r\|x^k-\tilde{x}^k\|^2\Big\}+ s_k\|y^k-\tilde{y}^k\|^2   \nonumber \\
 &\overset{\eqref{ccpd}}{\geq}&   r\|x^k-\tilde{x}^k\|^2 -\frac{1}{2}\Big\{\nu s_k\|y^k-\tilde{y}^k\|^2+r\|x^k-\tilde{x}^k\|^2\Big\}+ s_k\|y^k-\tilde{y}^k\|^2  \\
 &  = &    \frac{1}{2}\big\{r\|x^k-\tilde{x}^k\|^2+(2-\nu)s_k\|y^k-\tilde{y}^k\|^2\big\} \nonumber\\
 &\geq&  \frac{1}{2}\big\{r\|x^k-\tilde{x}^k\|^2+s_k\|y^k-\tilde{y}^k\|^2\big\}. \nonumber
\end{eqnarray}
This completes the proof of the lemma.
\end{proof}

Next, we show that $M_k^{PD}(w^k-\tilde{w}^k)$ is an ascent direction of the unknown distance function $\|w-w^\ast\|_{H^{PD}}^2$ at the point $w^k$.  To this end,
by setting $w$ in \eqref{APDA-Q} as any fixed $w^\ast\in\Omega^\ast$ and using the identity \eqref{EQF}, we get
\begin{eqnarray*}
% \nonumber to remove numbering (before each equation)
  \lefteqn{(\tilde{w}^{k}-w^\ast)^T Q_k^{PD}(w^k  -\tilde{w}^{k})\geq \theta(\tilde{w}^{k})-\theta(w^\ast)  + (\tilde{w}^{k}-w^\ast)^T F(\tilde{w}^{k})  } \\
   &=&     \theta(\tilde{w}^{k})-\theta(w^\ast)  + (\tilde{w}^{k}-w^\ast)^T F(w^\ast)\geq0. \qquad \qquad
\end{eqnarray*}
It follows from $\tilde{w}^{k}-w^\ast=(w^k-w^\ast)-(w^k-\tilde{w}^{k})$ that
\begin{eqnarray}\label{K-Dirpd}
% \nonumber to remove numbering (before each equation)
  \lefteqn{ (w^k-w^\ast)^T Q_k^{PD}(w^k  -\tilde{w}^{k})}  \nonumber \\
  &\geq& (w^k-\tilde{w}^k)^TQ_k^{PD}(w^k  -\tilde{w}^{k})\overset{\eqref{Key-inequality1}}{\geq}\frac{1}{2}\Big\{r\|x^k-\tilde{x}^k\|^2+s_k\|y^k-\tilde{y}^k\|^2\Big\}.
\end{eqnarray}
In light of the matrices $Q_k^{PD}$ and $H^{PD}$ defined in \eqref{QPD} are nonsingular,  the inequality \eqref{K-Dirpd} can be further reformulated as
\begin{eqnarray}\label{AD1}
% \nonumber to remove numbering (before each equation)
\lefteqn{\Big\langle \nabla(\frac{1}{2}\|w-w^\ast\|_{H^{PD}}^2)\Big|_{w=w^k}, (H^{PD})^{-1}Q_k^{PD}(w^k  -\tilde{w}^{k})\Big\rangle} \nonumber\\
  &\overset{\eqref{HMQ}}{=}& \Big\langle H^{PD}(w^k-w^\ast), M_k^{PD}(w^k  -\tilde{w}^{k})\Big\rangle  \geq\frac{1}{2}\Big\{r\|x^k-\tilde{x}^k\|^2+s_k\|y^k-\tilde{y}^k\|^2\Big\},
\end{eqnarray}
which indicates that $M_k^{PD}(w^k  -\tilde{w}^{k})$ is an ascent direction of the unknown distance function $\|w-w^\ast\|_{H^{PD}}^2$ at the point $w^k$. Consequently, we can update the new iterate $w^{k+1}$ by the correction step
\begin{equation}\label{cspd}
 w^{k+1} = w^k - \alpha M_k^{PD}(w^k  -\tilde{w}^{k}),
\end{equation}
where $\alpha>0$ is the corresponding step size. 

Similarly, let us turn to determine an appropriate step size $\alpha$ in the correction step \eqref{cspd} to make the new iterate $w^{k+1}$
 closer to $\Omega^\ast$ as much as possible.  Since
\begin{eqnarray}\label{cc-pd}
% \nonumber to remove numbering (before each equation)
  \lefteqn{\|w^{k+1}-w^\ast\|_{H^{PD}}^2} \nonumber \\
  &=& \|w^k - \alpha M_k^{PD}(w^k  -\tilde{w}^{k})-w^\ast\|_{H^{PD}}^2\\
  &=& \|w^k -w^\ast\|_{H^{PD}}^2-2\alpha(w^{k}-w^\ast)^TH^{PD}M_k^{PD}(w^k  -\tilde{w}^{k})+\alpha^2\|M_k^{PD}(w^k  -\tilde{w}^{k})\|_{H^{PD}}^2 \nonumber \\
  &\overset{\eqref{K-Dirpd}}{\leq}& \|w^k -w^\ast\|_{H^{PD}}^2-2\alpha(w^{k}-\tilde{w}^k)^TQ_k^{PD}(w^k  -\tilde{w}^{k})+\alpha^2\|M_k^{PD}(w^k  -\tilde{w}^{k})\|_{H^{PD}}^2, \nonumber
\end{eqnarray}
and we further define
$$q_k^{PD}(\alpha):=2\alpha(w^{k}-\tilde{w}^k)^TQ_k^{PD}(w^k  -\tilde{w}^{k})-\alpha^2\|M_k^{PD}(w^k  -\tilde{w}^{k})\|_{H^{PD}}^2.$$
By maximizing the quadratic term $q_k^{PD}(\alpha)$, we have
\begin{equation}\label{pd-stepsize}
  \alpha_k^\ast =\frac{(w^{k}-\tilde{w}^k)^TQ_k^{PD}(w^k  -\tilde{w}^{k})}{\|M_k^{PD}(w^k  -\tilde{w}^{k})\|_{H^{PD}}^2}.
\end{equation}
Similarly, due to $q_k^{PD}(\alpha)$ is a lower bound  quadratic contraction function, we further introduce a relaxation factor $\gamma\in (0, 2)$ and set $\alpha_k=\gamma\alpha_k^\ast$, which immediately yields the tailored correction step \eqref{APDA-c}.

\begin{lemma}\label{lemmaPD}
Let $\{w^k\}$ and $\{\tilde{w}^k\}$ be the sequences generated by Algorithm \ref{alg1} and the relaxed condition \eqref{ccpd} hold, and let $\underline{s}$ be the lower bound of the sequence $\{s_k\}$. Then, for any $\gamma\in(0,2)$ and $w^\ast\in\Omega^\ast$, there exists a constant $C^{PD}>0$ such as
\begin{equation}\label{con-con-pd}
   \|w^k -w^\ast\|_{H^{PD}}^2-\|w^{k+1}-w^\ast\|_{H^{PD}}^2\geq \gamma(2-\gamma)C^{PD}\|w^k-\tilde{w}^k\|_{\underline{H}^{PD}}^2,
\end{equation}
where
\begin{equation}\label{Hlpd}
\underline{H}^{PD}=\left(
  \begin{array}{cc}
    r I_n  & 0 \\
   0   & \underline{s}I_m \\
  \end{array}
\right).
\end{equation}
\end{lemma}
\begin{proof}
To begin with, by setting $\alpha=\gamma\alpha_k^\ast$ in \eqref{cc-pd}, we have 
\begin{eqnarray*}
% \nonumber to remove numbering (before each equation)
  \lefteqn{\|w^{k+1}-w^\ast\|_{H^{PD}}^2}\\
  &\leq& \|w^k -w^\ast\|_{H^{PD}}^2-2\gamma\alpha_k^\ast(w^{k}-\tilde{w}^k)^TQ_k^{PD}(w^k  -\tilde{w}^{k})+\gamma^2(\alpha_k^\ast)^2\|M_k^{PD}(w^k  -\tilde{w}^{k})\|_{H^{PD}}^2 \\
  &\overset{\eqref{pd-stepsize}}{\leq}& \|w^k -w^\ast\|_{H^{PD}}^2-2\gamma\alpha_k^\ast(w^{k}-\tilde{w}^k)^TQ_k^{PD}(w^k  -\tilde{w}^{k})+\gamma^2\alpha_k^\ast(w^{k}-\tilde{w}^k)^TQ_k^{PD}(w^k  -\tilde{w}^{k}) \\
  &=& \|w^k -w^\ast\|_{H^{PD}}^2 - \gamma(2-\gamma)\alpha_k^\ast(w^{k}-\tilde{w}^k)^TQ_k^{PD}(w^k  -\tilde{w}^{k}).
\end{eqnarray*}
Moreover, it follows from  \eqref{Key-inequality1} that 
\begin{eqnarray*}
% \nonumber to remove numbering (before each equation)
  &&\|w^k -w^\ast\|_{H^{PD}}^2-\|w^{k+1}-w^\ast\|_{H^{PD}}^2  \geq  \gamma(2-\gamma)\alpha_k^\ast(w^{k}-\tilde{w}^k)^TQ_k^{PD}(w^k  -\tilde{w}^{k})  \\
   &&\geq \frac{1}{2}\gamma(2-\gamma)\alpha_k^\ast\big\{r\|x^k-\tilde{x}^k\|^2+s_k\|y^k-\tilde{y}^k\|^2\big\} \\
   &&\geq\frac{1}{2} \gamma(2-\gamma)\alpha_k^\ast \big\{r\|x^k-\tilde{x}^k\|^2+\underline{s}\|y^k-\tilde{y}^k\|^2\big\}  \\
   &&=\frac{1}{2}\gamma(2-\gamma)\alpha_k^\ast \|w^k-\tilde{w}^k\|_{\underline{H}^{PD}}^2.
\end{eqnarray*}
It is sufficient to show that there exists a constant $C^{PD}$ satisfying $\alpha_k^\ast\geq 2C^{PD}$ uniformly. Recall the matrices $H^{PD}$ and $M_k^{PD}$ defined in \eqref{QPD} and \eqref{MPD} respectively, and that the sequence $\{s_k\}$ has an upper bound $\overline{s}$ given by \eqref{ulpd}. We have
$$0\prec (M_k^{PD})^TH^{PD}M_k^{PD}=\left(\!\!
  \begin{array}{cc}
    r I_n & A^T \\
    A   & \frac{1}{r} A A^T + \frac{s_k^2}{s_a}I_m \\
  \end{array}\!\!
\right)\preceq\left(\!\!
  \begin{array}{cc}
    r I_n & A^T \\
    A   & \frac{1}{r} A A^T + \frac{\overline{s}^2}{s_a}I_m \\
  \end{array}\!\!
\right)=:\overline{H}^{PD}.$$
It further implies that
\begin{eqnarray*}
% \nonumber to remove numbering (before each equation)
 \alpha_k^\ast &=&\frac{(w^{k}-\tilde{w}^k)^TQ_k^{PD}(w^k  -\tilde{w}^{k})}{\|M_k^{PD}(w^k  -\tilde{w}^{k})\|_{H^{PD}}^2}
  \geq \frac{\|w^k-\tilde{w}^k\|_{\underline{H}^{PD}}^2}{2\|w^k  -\tilde{w}^{k}\|_{\overline{H}^{PD}}^2}.
\end{eqnarray*}
The assertion of lemma follows immediately by using the norm equivalence principle.
\end{proof}

With the assertion of Lemma \ref{lemmaPD}, we can immediately obtain the global convergence theory of the proposed scheme, which is summerized as the following theorem.
\begin{theorem}
Let $\{w^k \}$ and $\{\tilde{w}^k\}$  be the sequences generated by Algorithm \ref{alg1} and the relaxed condition \eqref{ccpd} hold. Then, the sequence  $\{w^k \}$  converges to some $w^\infty\in\Omega^\ast$.
\end{theorem}
\begin{proof}
To begin with, it follows from \eqref{con-con-pd} that the generated sequence $\{w^k\}$ is bounded. Let $w^{\infty}$ be a cluster point of $\{w^k\}$ and $\{w^{k_j}\}$ be a subsequence converging to $w^{\infty}$.  By summing the inequality  \eqref{con-con-pd} over $k=0,1,\ldots,\infty$, we obtain
$$\sum_{k=0}^{\infty}\|w^k -\tilde{w}^k\|_{\underline{H}^{PD}}^2\leq\frac{1}{\gamma(2-\gamma)C^{PD}}\|w^0 -w^\ast\|_{H^{PD}}^2,$$
which further implies
\begin{equation}\label{Contrac-uut-pd}
  \lim_{k\to \infty}\|w^k -\tilde{w}^k\|_{\underline{H}^{PD}}=0.
\end{equation}
Moreover, it follows from \eqref{Contrac-uut-pd} that the sequence $\{\tilde{w}^{k_j}\}$ also converges to $w^{\infty}$. Then, according to \eqref{APDA-Q}, we have
$$\tilde{w}^{k_j}\in \Omega, \quad \theta(w)-\theta(\tilde{w}^{k_j}) + (w-\tilde{w}^{k_j})^TF(\tilde{w}^{k_j}) \ge (w-\tilde{w}^{k_j})^TQ_k^{PD}(w^{k_j}-\tilde{w}^{k_j}), \quad \forall\; w\in \Omega.$$
Using the continuity of $\theta(w)$ and $F(w)$, we have
$$w^{\infty}\in \Omega, \quad \theta(w)-\theta(w^{\infty}) + (w- w^{\infty})^T F(w^{\infty}) \ge 0, \quad \forall\; w\in \Omega. $$
This means that  $w^{\infty}$ is a solution point of the VI \eqref{VI}, which is also a saddle point of the studied model \eqref{Min-Max}.  Furthermore, it follows from \eqref{con-con-pd} that
$$\|w^{k+1} - w^{\infty}\|_{H^{PD}} \le \|w^k  - w^{\infty}\|_{H^{PD}},$$
which indicates that the sequence $\left\{\left\|w^{k}-w^{\infty}\right\|_{H^{PD}}\right\}_{k \geq 0}$ is nonincreasing and  it is thus convergent. Moreover, it follows from  $\lim_{j\rightarrow\infty}\|w^{k_j}-w^{\infty}\|_{H^{PD}}=0$  that the sequence $\{w^k\}$ also converges to $w^\infty$. This completes the proof of the theorem.
\end{proof}

\section{Numerical experiments}\label{sec5}
\setcounter{equation}{0}
\setcounter{remark}{0}
In this section, we report the numerical performance of the proposed adaptive algorithm on the challenging assignment problem. The preliminary numerical results affirmatively shows that the proposed algorithm can achieve a conspicuous acceleration. Our codes were written in Python 3.11 and they were implemented in a laptop with  2.40 GHz Intel Core i7-13700H CPU and 16 GB memory.
\subsection{Assignment problem and its saddle point reformulation}\label{sec-assignment}

The assignment problem is a fundamental model in operational research, and it aims at assigning $n$ jobs to $n$ operators to get the maximum profit. As discussed in e.g., \cite{LYe}, the assignment problem can be stated as
\begin{eqnarray}\label{assignment}
% \nonumber to remove numbering (before each equation)
   && \max  z= \sum_{i=1}^n  \sum_{j=1}^n   c_{ij} x_{ij} \nonumber  \\[0.1cm]
   && \hbox{\textbf{s.t.}}\left\{
                            \begin{array}{ll}
                             \sum_{j=1}^n   x_{ij}  = 1, \quad i=1, \ldots, n, \\[0.3cm]
                             \sum_{i=1}^n   x_{ij}  = 1,  \quad j=1, \ldots, n,  \\[0.3cm]
                             x_{ij} \in \{ 0, 1\},
                            \end{array}
                          \right.
\end{eqnarray}                            
in which $c_{ij}>0$ represents the benefit  of assigning the $j$-th job to the $i$-th operator. As discussed in \cite{HeMXY}, the model \eqref{assignment} can be equivalently relaxed as the following convex continuous model in which the binary constraints are replaced by box constraints:
\begin{eqnarray}\label{Assignment-R}
% \nonumber to remove numbering (before each equation)
   && \max  z= \sum_{i=1}^n  \sum_{j=1}^n   c_{ij} x_{ij} \nonumber  \\[0.1cm]
   && \hbox{\textbf{s.t.}}\left\{
                            \begin{array}{ll}
                             \sum_{j=1}^n   x_{ij}  = 1, \quad i=1, \ldots, n, \\[0.3cm]
                             \sum_{i=1}^n   x_{ij}  = 1,  \quad j=1, \ldots, n,  \\[0.3cm]
                             0 \leq x_{ij} \leq 1.
                            \end{array}
                          \right.
\end{eqnarray}   
Moreover, let $e_n\in\Re^n$ be the vector whose elements are all 1 and $I_n\in\Re^{n\times n}$ be the identity matrix. As discussed in \cite{HeMXY}, the studied model \eqref{Assignment-R} can be further rewritten compactly as the following convex programming problem with linear equality constraints 
\begin{equation}\label{A-COP}
 \min\big\{-c^Tx \;|\; Ax=e_{2n},\;x\in \mathcal{X}\big\},
\end{equation}
in which 
   $$x = (x_{11}, x_{12}, \ldots, x_{1n}, x_{21}, x_{22}, \ldots, x_{2n}, \cdots, x_{n1}, x_{n2}, \ldots, x_{nn})^T,$$
   $$c = (c_{11}, c_{12}, \ldots, c_{1n}, c_{21}, c_{22}, \ldots, c_{2n}, \cdots, c_{n1}, c_{n2}, \ldots, c_{nn})^T,$$
   $$\mathcal{X}=\big\{x\in\Re^{n^2}\;|\;0\leq x_{ij}\leq1;\; i=1,\ldots,n, \,j=1,\ldots,n\big\},$$
and the constrained coefficient matrix is defined as   
\begin{equation}\label{A-A}
 A=\left( \begin{array}{ccccc}
                   \quad e_n^T  \quad   &            &                             &               &      \\[0.05cm]
                               &   \quad e_n^T  \quad   &                             &                &       \\
                               &            &  \quad  \cdots \; \; \cdots  \quad   &                &       \\
                              &            &                              &   \quad  e_n^T  \quad     &      \\[0.05cm]
                                &          &                            &                    &     \quad e_n^T   \quad       \\[0.2cm]
             \hbox{\large $ I_n$}  &  \hbox{\large $ I_n$}  &    \;  \cdots \;\;  \cdots   & \hbox{\large $ I_n$}   &  \hbox{\large $ I_n$}
              \end{array} \right).
\end{equation}
By imposing the corresponding Lagrange function to \eqref{A-COP}, the equivalent  saddle-point representation of \eqref{A-COP} can be stated as  
\begin{equation}\label{A-SPP}
  \min_{x\in \cal{X}} \max_{y\in \Re^{2n}} L(x,y) := (-c^Tx) -  y^TAx - (-e_{2n}^Ty).
\end{equation}

Let us turn to calculate the spectrum and average spectrum of the matrix $A^T\!A$. It follows from the fundamental results of linear algebra that 
$$ \rho(A^T\!A) =  \rho(AA^T) \quad \hbox{and} \quad  \hbox{Trace}(A^T\!A)    = \hbox{Trace}(AA^T).$$
 Note that
\[\label{AAT}
 AA^T  =
    \left(\begin{array}{cccccccccc}
                 n     & 0        & \ldots   &   \ldots  &   0           &          1   &  1   & \ldots &\ldots   & 1       \\[-0.1cm]
                 0     &  n       & \ddots  &                &  \vdots   &          1   &  1   & \ldots & \ldots   & 1      \\[-0.1cm]
           \vdots &\ddots &\ddots   &  \ddots  & \vdots    &  \vdots  &\vdots  &  &   & \vdots            \\[-0.1cm]
           \vdots &             &\ddots   &  \ddots  &        0      &   1 &  1   & \ldots & \ldots   & 1               \\[-0.1cm]
                0    &  \ldots &\ldots    &       0       &        n      &   1 &  1   & \ldots & \ldots   & 1               \\[-0.1cm]
                  1   &  1   & \ldots &\ldots   & 1    &     n     & 0        & \ldots   &   \ldots  &   0        \\[-0.1cm]
                            1   &  1   & \ldots & \ldots   & 1   & 0     &  n       & \ddots  &                &  \vdots           \\[-0.1cm]
                   \vdots  &\vdots  &  &   & \vdots   &   \vdots &\ddots &\ddots   &  \ddots  & \vdots           \\[-0.1cm]
                     1 &  1   & \ldots & \ldots   & 1   &  \vdots &             &\ddots   &  \ddots  &        0                \\[-0.1cm]
                      1 &  1   & \ldots & \ldots   & 1     &    0    &  \ldots &\ldots    &       0       &        n
                  \end{array} \right)            =
                      \left(\begin{array}{cc}
                        nI_n     & e_ne_n^T  \\
                        e_ne_n^T      &    nI_n  \end{array} \right) .
        \]
According to the same analysis in \cite{HeMXY}, we have 
$$ \rho(A^T\!A) = \rho(AA^T) =2n. $$
For the average eigenvalue of $A^T\!A$, it follows from \eqref{AAT} that
$$\hbox{Trace}(A^T\!A)= \hbox{Trace}(AA^T)  = 2n^2.$$
Then we have
$$\rho_{\hbox{\footnotesize average}}(A^T\!A) =\frac{2n^2}{n^2}= 2 \quad \hbox{and} \quad \rho_{\hbox{\footnotesize average}}(A\!A^T) =\frac{2n^2}{2n} = n.$$
In light of $\rho_{\hbox{\footnotesize average}}(A^T\!A)=2\ll 2n=\rho(A^T\!A)$ for a large $n$, we adopt the adaptive DPHG (i.e., Algorithm \ref{alg2}) in our experiments.

\subsection{Algorithmic settings and implementation}

To simulate the assignment problem, we follow the parameter settings in \cite{HeMXY} and generate the tested data for the model \eqref{Assignment-R} with different values of $n$. Specifically,  we choose $ c_{ij} = \hbox{random} \times  10$ for $ i=1, \ldots, n,  j=1, \ldots, n$; the initials $x_{ij}^0 = \frac{1}{n}$ for $i=1,\ldots, n, \; j=1, \ldots, n$ and $y_l^0= 0$ for $l=1,\ldots,2n$.  In our experiments, we take the Chambolle-Pock's primal dual algorithm  (abbreviated as ``CP-PDA") as the benchmark for comparison. At the same time, we also compare the efficient heuristic PDA with fixed  step size based on average spectrum proposed in \cite{HeMXY}, despite its convergence theory is still missing.

For implementing the tested algorithms more efficiently, we follow the parameter settings in \cite{HeMXY} and the concrete parameters for various algorithms are chosen as follows:
\begin{itemize}
\item The CP-PDA \eqref{C-P}: $r=(10/n) \times \sqrt{n/2}$ and $s = 0.4 n \times  \sqrt{n/2}$;
\item The heuristic PDA: $r=10/n$ and $s= 4/r= 0.4 n$;
\item  Adaptive DPHG: $s=n$, $r_0=3/s$, $r_a=10/s$, $\underline{r}=r_a/\sqrt{n}$, $\nu=0.9$, $\mu=0.5$,  $\theta=1.2$ and  $\gamma=1$.
 \end{itemize}
 Moreover, the stopping criterion for \eqref{A-SPP} is defined as 
$$\frac{\|w^k-w^{k-1}\|}{\|w^k\|} <10^{-10}$$
for all algorithms tested.

\subsection{Numerical results}

In Table \ref{Table1}, we list the computational results of the CP-PDA \eqref{C-P}, the  heuristic PDA and the adaptive DPHG for the assignment problem \eqref{Assignment-R}. The objective function values (denoted by `$\Phi(x^k)$'), required iterations (denoted by `Iter.') and the computational time in seconds (denoted by `CPU(s)') to the stopping criteria are reported. It is clear that the proposed adaptive DPHG has a superior numerical performance compared with the benchmark CP-PDA in both the required iterations and computing time. The acceleration of the adaptive scheme over the Chambolle-Pock's primal dual algorithm can be as much as 10 times or more. Moreover, the proposed adaptive DPHG performs competitively compared with the heuristic PDA, despite the convergence theory of the later one is still missing.
 \begin{figure}[H]
\centering
\subfigure[$n=100$]{
\includegraphics[width=8.3cm]{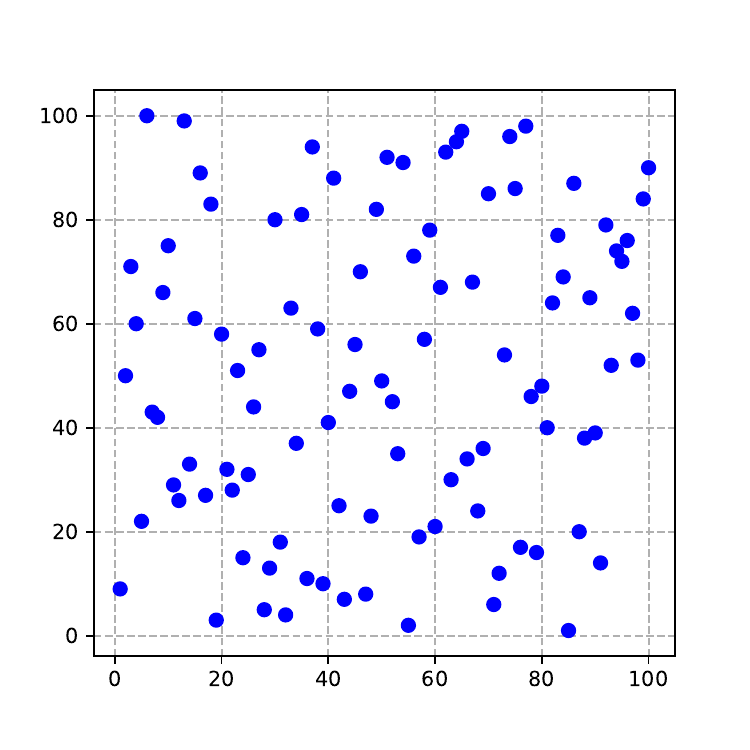}%设置插入的图大小和文件名，包括JPG,PNG,PDF,EPS等，放在源文件目录下
}\hspace{-8mm}
\subfigure[$n=200$]{
\includegraphics[width=8.3cm]{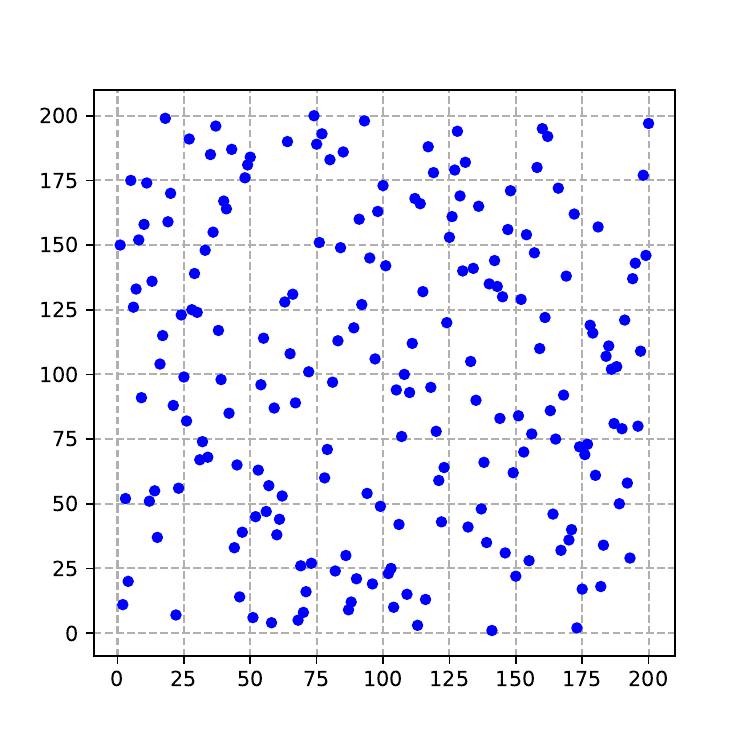}%设置插入的图大小和文件名，包括JPG,PNG,PDF,EPS等，放在源文件目录下
}\\ \vspace{-0.3cm}
\subfigure[$n=500$]{
\includegraphics[width=8.3cm]{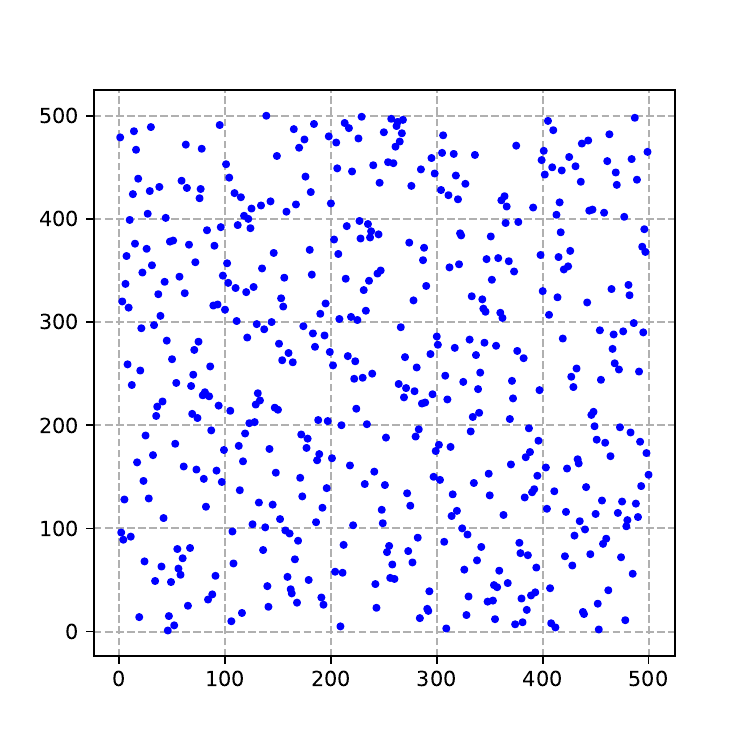}%设置插入的图大小和文件名，包括JPG,PNG,PDF,EPS等，放在源文件目录下
}\hspace{-8mm}
\subfigure[$n=1000$]{
\includegraphics[width=8.3cm]{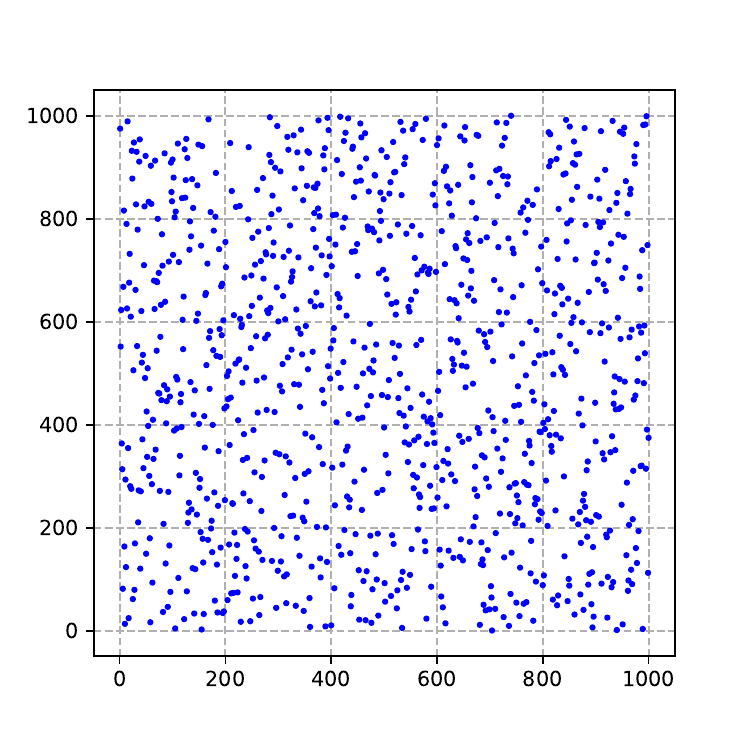}%设置插入的图大小和文件名，包括JPG,PNG,PDF,EPS等，放在源文件目录下
}
\caption{Computational results of the adaptive DPHG for the assignment problem \eqref{Assignment-R} with various sizes $n$. }
\label{fig1}
\end{figure}
 
 At the same time, it can be empirically verified that solutions of the relaxed problem \eqref{Assignment-R}  are all binary. Note that the convergence theory of the heuristic PDA is missing.
 The adaptive DPHG hence  provides a very efficient and reliable  solver to the challenging assignment problem. Moreover,  we visualize the solutions for some tested scenarios for the cases $n=100,\, 200,\,500,\,1000$ in Figure \ref{fig1}.

 \begin{table}[H]
\caption{Numerical results of  the  assignment problem \eqref{Assignment-R} with various $n$. }
\centering
\setlength{\tabcolsep}{2.2mm}{
\begin{tabular}{c ccccccc ccc}
  \toprule
 \multirow{2}{*}{Size $n$} &   \multicolumn{3}{c}{CP-PDA} &   \multicolumn{3}{c}{Heuristic PDA}  & \multicolumn{3}{c}{Adaptive DPHG}  \cr \cmidrule(lr){2-4} \cmidrule(lr){5-7} \cmidrule(lr){8-10}
                                &     Iter.        &   CPU(s)    & $\Phi(x^k)$ &   Iter.         &    CPU(s)         &   $\Phi(x^k)$      &   Iter.         &    CPU(s)         &   $\Phi(x^k)$      \cr
  \midrule
                $100$      &    1053       &    0.08        &   983.74      &   144     &    0.01    &    983.74     &    256      &    0.03       &   983.74      \\
                $200$      &    2139       &    0.66        &   1983.47    &   200     &    0.06    &    1983.47   &    290      &    0.17       &   1983.47    \\
                $300$      &    2407       &    1.67        &   2983.28    &   217     &    0.21    &    2983.28   &    372      &    0.38       &   2983.28    \\
                $400$      &    10330     &    40.43      &   3982.58    &   699     &    2.67    &    3982.58   &    879      &    4.76       &   3982.58    \\
                $500$      &    4739       &    28.62      &   4983.45    &   296     &    1.72    &    4983.45   &    504      &    4.27       &   4983.45    \\
                $600$      &    9321       &    86.74      &   5983.80    &   505     &    4.13    &    5983.80   &    601      &    7.35       &   5983.80    \\
                $700$      &    18774     &    241.23    &   6984.00    &   1041   &    11.60  &    6984.00   &    1120    &    20.05     &   6984.00     \\
                $800$      &    20660     &    361.12    &   7983.59    &   1034   &    15.02  &    7983.59   &    1224    &    28.60     &   7983.59     \\
                $900$      &    9500       &    204.58    &   8983.77    &   443     &    8.18    &    8983.77   &    635      &    18.18     &   8983.77     \\
                $1000$    &    16699     &    429.32    &   9984.04    &   704     &    15.92  &    9984.04   &    941      &    33.50     &   9984.04      \\
  \bottomrule
 \end{tabular}
 \label{Table1}}
\end{table}

%\subsection{Linear programming problem}
%Consider the linear programming problem
%\begin{equation}\label{LP3}
%\min_{x_1,x_2}\{x_1+2x_2  \;|\;  x_1+x_2=1;\;x_1\geq0,\,x_2\geq0\}.
%\end{equation}
%The specific parameter settings for the tested algorithm are chose as follows:
%\begin{itemize}
%\item $\rho_{\hbox{\footnotesize average}}(A^T\!A) = 1 \quad \hbox{and} \quad \rho_{\hbox{\footnotesize average}}(A\!A^T) = 2.$
%\item  Adaptive PDHG-1: $r=  2$ and $s_0=1.5/r$ .
%\item  Adaptive PDHG-2: $s=  1$ and $r_0=1.5/s$ .
% \end{itemize}
%
%
%\begin{figure}[H]
%  \centering
%  \includegraphics[width=9.0cm]{Figure/LP-1}
%  \caption{The orbit of the adaptive PDHG-1 (35 iterations)}
%\end{figure}
%
%\begin{figure}[H]
%  \centering
%  \includegraphics[width=9.0cm]{Figure/LP-2}
%  \caption{The orbit of the adaptive PDHG-2 (5 iterations)}
%\end{figure}

\section{Conclusions}\label{sec6}
In this paper, we present a class of  adaptive primal dual type algorithms based on the associated average spectrum for the generic convex saddle point problem. For the case where $\rho_{\hbox{\footnotesize average}}(A^T\!A)\ll \rho(A^T\!A)$ and the primal subproblem can be implemented easily, we propose an adaptive dual primal hybrid gradient algorithm which fixes the dual regularization parameter and tunes the primal regularization parameter adaptively. Alternatively, for the case where $\rho_{\hbox{\footnotesize average}}(AA^T)\ll \rho(AA^T)$ and the dual subproblem can be tackled easily, we propose an adaptive primal dual hybrid gradient algorithm which fixes the primal regularization parameter and adjusts the dual regularization parameter dynamically.  By exploiting the prediction-correction algorithmic framework, the global convergence theory of the  proposed schemes is determined only by some more  relaxed conditions regarding the average spectrum. In particular, the proposed adaptive dual primal hybrid gradient algorithm provides an efficient and reliable solver to the challenging assignment problem. In the future work, we will explore the potential high-order convergence rate of the novel methods.  

%we will explore the adaptive scheme which can adjust the primal and dual step sizes dynamically and simultaneously. 

\end{CJK*}
\end{document}